\documentclass[notitlepage]{scrartcl}
\usepackage[short labels]{enumitem}
\usepackage{amsmath}
\usepackage{amssymb}
\usepackage{mathtools}

\usepackage{amsthm}
\usepackage{abstract}
\usepackage{graphicx}
\usepackage{caption}
\usepackage{subcaption}
\graphicspath{ {images_for_irregular/} }
\usepackage{xcolor}
\usepackage{comment}
\usepackage{float}
\usepackage{dsfont}
\usepackage[font=small,labelfont=bf]{caption}
\usepackage{hyperref}
\hypersetup{
    colorlinks=true,
    linkcolor=blue,
    filecolor=magenta,      
    urlcolor=cyan
    }
\makeatletter
\newcommand*{\barfix}[2][.175ex]{%
  \mathpalette{\@barfix{#1}}{#2}%
}
\newcommand*{\@barfix}[3]{%
  \vbox{%
    \kern#1\relax
    \hbox{$#2#3\m@th$}%
  }%
}
\makeatother

\newtheorem{thm}{Theorem}

\newtheorem{lemma}[thm]{Lemma}
\newtheorem{claim}[thm]{Claim}

\newtheorem{question}[thm]{Question}
\newtheorem{conjecture}[thm]{Conjecture}
\newcommand{\footremember}[2]{%
    \footnote{#2}
    \newcounter{#1}
    \setcounter{#1}{\value{footnote}}%
}

\usepackage[margin=1in]{geometry}

\title{Percolation on Irregular High-dimensional Product Graphs}
\author{%
Sahar Diskin \footremember{alley}{School of Mathematical Sciences, Tel Aviv University, Tel Aviv 6997801, Israel. Email: sahardiskin@mail.tau.ac.il.}%
\and Joshua Erde \footremember{trailer}{Institute of Discrete Mathematics, Graz University of Technology, Steyergasse 30, 8010 Graz, Austria. Email: erde@math.tugraz.at. Research supported in part by FWF P36131.}%
\and Mihyun Kang \footremember{alley2}{Institute of Discrete Mathematics, Graz University of Technology, Steyergasse 30, 8010 Graz, Austria. Email: kang@math.tugraz.at. Research supported in part by FWF W1230 and FWF I6502. Part of this work was done while the author was visiting the Simons Institute for the Theory of Computing.}%
\and Michael Krivelevich \footremember{trailer2}{School of Mathematical Sciences, Tel Aviv University, Tel Aviv 6997801, Israel. Email:
krivelev@tauex.tau.ac.il. Research supported in part by USA–Israel BSF grant 2018267.}%
}
\begin{document}
\maketitle

\begin{abstract}
We consider bond percolation on high-dimensional product graphs $G=\square_{i=1}^tG^{(i)}$, where $\square$ denotes the Cartesian product. We call the $G^{(i)}$ the base graphs and the product graph $G$ the host graph. Very recently, Lichev \cite{L22} showed that, under a mild requirement on the isoperimetric properties of the base graphs, the component structure of the percolated graph $G_p$ undergoes a phase transition when $p$ is around $\frac{1}{d}$, where $d$ is the average degree of the host graph.

In the supercritical regime, we strengthen Lichev's result by showing that the giant component is in fact unique, with all other components of order $o(|G|)$, and determining the sharp asymptotic order of the giant. Furthermore, we answer two questions posed by Lichev \cite{L22}: firstly, we provide a construction showing that the requirement of bounded-degree is necessary for the likely emergence of a linear order component; secondly, we show that the isoperimetric requirement on the base graphs can be, in fact, super-exponentially small in the dimension. 
Finally, in the subcritical regime, we give an example showing that in the case of \textit{irregular} high-dimensional product graphs, there can be a \textit{polynomially} large component with high probability, very much unlike the quantitative behaviour seen in the Erd\H{o}s-R\'enyi random graph and in the percolated hypercube, and in fact in any \textit{regular} high-dimensional product graphs, as shown by the authors in a companion paper \cite{DEKK22}.
\end{abstract}

\section{Introduction and main results}
In 1957, Broadbent and Hammersley \cite{BH57} initiated the study of percolation theory in order to model the flow of fluid through a medium with randomly blocked channels. In \emph{(bond) percolation}, given a graph $G$, the percolated random subgraph $G_p$ is obtained by retaining every edge of the \emph{host graph} $G$ independently and with probability $p$. If we take the host graph $G$ to be the complete graph $K_{d+1}$ then $G_p$ coincides with the well-known \emph{binomial random graph model} $G(d+1,p)$. In their seminal paper from 1960, Erd\H{o}s and R\'enyi \cite{ER60} showed that $G(d+1,p)$\footnote{In fact, Erd\H{o}s and R\'enyi worked in the closely related \emph{uniform} random graph model $G(d+1,m)$. As we mainly consider graphs with average degree $d$, we use the slightly unusual notation of $G(d+1,p)$ instead of $G(n,p)$, to make the comparison of the results simpler.} undergoes a dramatic phase transition, with respect to its component structure, when $p$ is around $\frac{1}{d}$.  More precisely, given a constant $\epsilon>0$ let us define $y=y(\epsilon)$ to be the unique solution in $(0,1)$ of the equation
\begin{align}\label{survival prob}
    y=1-\exp\left(-(1+\epsilon)y\right),
\end{align}
where we note that $y$ is an increasing continuous function on $(0,\infty)$ with $y(\epsilon) = 2\epsilon - O(\epsilon^2)$.
\begin{thm}[\cite{ER60}]\label{ER thm}
Let $\epsilon>0$ be a small enough constant. Then, with probability tending to one as $d \to \infty$,
\begin{itemize}
    \item [(a)] if $p=\frac{1-\epsilon}{d}$, then all the components of $G(d+1,p)$ are of order $O_{\epsilon}\left(\log d \right)$; and,
    \item [(b)] if $p=\frac{1+\epsilon}{d}$, then there exists a unique \emph{giant} component in $G(d+1,p)$ of order $(1+o(1))y(\epsilon)d$. Furthermore, all the other components of $G(d+1,p)$ are of order $O_{\epsilon}\left(\log d\right)$.
\end{itemize}
\end{thm}
We refer the reader to \cite{B01, FK16, JLR00} for a systematic coverage of random graphs, and to the monographs \cite{BR06, G99, K82} on percolation theory.

This phenomenon of such a sharp change in the order of the largest component has been subsequently studied in many other percolation models. Some well-studied examples come from percolation on lattice-like structures with fixed dimension (see \cite{HH17} for a survey on many results in this subject). Another extensively studied model is the percolated hypercube, where the host graph is the \emph{$d$-dimensional hypercube} $Q^d$. Indeed, answering a question of Erd\H{o}s and Spencer \cite{ES79}, Ajtai, Koml\'os, and Szemer\'edi \cite{AKS81} proved that $Q^d_p$ undergoes a phase transition quantitatively similar to the one which occurs in $G(d+1,p)$, and their work was later extended by Bollob\'as, Kohayakawa, and \L{}uczak \cite{BKL92}.  
\begin{thm}[\cite{AKS81, BKL92}] \label{hypercube thm}
Let $\epsilon>0$ be a small enough constant. Then, with probability tending to one as $d \to \infty$,
\begin{itemize}
    \item [(a)] if $p=\frac{1-\epsilon}{d}$, then all the components of $Q^d_p$ are of order $O_{\epsilon}(d)$; and,
    \item [(b)] if $p=\frac{1+\epsilon}{d}$, then there exists a unique giant component of order $(1+o(1))y(\epsilon)2^d$. Furthermore, all the other components of $Q^d_p$ are of order $O_{\epsilon}(d)$.
\end{itemize}
\end{thm}
Note that, since $|V(Q^d)| = 2^d$, in both Theorems \ref{ER thm} and \ref{hypercube thm} the likely order of the largest component changes from logarithmic in the order of the host graph in the subcritical regime, to linear in the order of the host graph in the supercritical regime, whilst the second-largest component in the supercritical regime remains of logarithmic order. Furthermore, in both models this giant component is the unique component of linear order, and covers the same asymptotic fraction of the vertices in each case. We informally refer to this quantitative behaviour as the \emph{Erd\H{o}s-R\'enyi component phenomenon}.

Recently, Lichev \cite{L22} initiated the study of percolation on some families of high-dimensional graphs, those arising from the product of many bounded-degree graphs. Given a sequence of graphs $G^{(1)},\ldots, G^{(t)}$, the Cartesian product of $G^{(1)},\ldots, G^{(t)}$, denoted by $G=G^{(1)}\square \cdots \square G^{(t)}$ or $G=\square_{i=1}^{t}G^{(i)}$, is the graph with the vertex set
\begin{align*}
    V(G)=\left\{v=(v_1,v_2,\ldots,v_t) \colon v_i\in V(G^{(i)}) \text{ for all } i \in [t]\right\},
\end{align*}
and the edge set
\begin{align*}
   E(G)=\left\{uv \colon \begin{array}{l} \text{there is some } i\in [t] \text{ such that }  u_iv_i\in E\left(G^{(i)}\right)\\  \text{ and } u_j=v_j  
    \text{ for all } i \neq j  \end{array}\right\}.
\end{align*}
We call $G^{(i)}$ the \textit{base graphs of} $G$. Throughout the rest of the paper, we denote by $|G|$ the number of vertices of the graph $G$, and use $t$ for the number of base graphs in the product. We denote by $d\coloneqq d(G)$ the average degree of a given graph $G$.

Considering percolation in these high-dimensional product graphs, Lichev \cite{L22} showed the existence of a threshold for the appearance of a component of linear order in these models, under some mild assumptions on the \emph{isoperimetric constants} of the base graphs. The isoperimetric constant $i(H)$ of a graph $H$ is a measure of the edge-expansion of $H$ and is given by
\[
i(H)=\min_{\begin{subarray}{c}S\subseteq V(H),\\
|S|\le |V(H)|/2\end{subarray}}\frac{e(S, S^C)}{|S|}.
\]

Throughout the paper, all asymptotics are with respect to $t$, and we use the standard Landau notations (see e.g., \cite{JLR00}).  We often state results for properties that hold \textbf{whp}, (short for \lq\lq with high probability\rq\rq), that is, with probability tending to $1$ as $t$ tends to infinity.

\begin{thm}[{\cite{L22}}]\label{product graphs}
Let $C, \gamma>0$ be constants and let $\epsilon>0$ be a small enough constant. Let $G^{(1)},\ldots, G^{(t)}$ be graphs such that $1\le \Delta\left(G^{(j)}\right)\le C$ and $i\left(G^{(j)}\right)\ge t^{-\gamma}$ for all $j\in[t]$. Let $G=\square_{j=1}^{t}G^{(j)}$. Then, \textbf{whp} 
\begin{enumerate}[(a)]
    \item\label{i:subcritical} if $p=\frac{1-\epsilon}{d}$, then all the components of $G_p$ are of order at most $\exp\left(-\frac{\epsilon^2t}{9C^2}\right)|G|$;
    \item\label{i:supercritical} if $p=\frac{1+\epsilon}{d}$, then there exists a positive constant $c_1=c_1(\epsilon,C,\gamma)$ such that the largest component of $G_p$ is of order at least $c_1|G|$.
\end{enumerate}
\end{thm}

Lichev asked if the conditions in Theorem \ref{product graphs} could be weakened.

\begin{question} [{\cite[Question 5.1]{L22}}]\label{q:maxdeg}
Does Theorem \ref{product graphs} still hold without the assumption on the maximum degrees of the $G^{(j)}\mathord{?}$
\end{question}

\begin{question} [{\cite[Question 5.2]{L22}}]\label{q:expdecay}
Does Theorem \ref{product graphs} still hold if the isoperimetric constant $i\left(G^{(j)}\right)$ decreases faster than a polynomial function of $t\mathord{?}$
\end{question}

Furthermore, in comparison to Theorems \ref{ER thm} and \ref{hypercube thm}, we note that Theorem \ref{product graphs} only gives a rough, qualitative description of the phase transition, and it is natural to ask if a more precise, quantitative description of the component structure of these percolated product graphs in the sub- and super-critical regimes can be given, in the vein of the Erd\H{o}s-R\'enyi component phenomenon, and if not in general, then under which additional assumptions?

In a recent paper \cite{DEKK22}, the authors gave a partial answer to this final question, showing that it is sufficient to assume that the base graphs are all \emph{regular} and of bounded order.

\begin{thm}[Informal {\cite{DEKK22}}]\label{regular product graphs}
Let $C>1$ be a constant and let $G=\square_{j=1}^{t}G^{(j)}$ be a product graph where $G^{(j)}$ is connected and regular and $1 < \left|V\left(G^{(j)}\right)\right| \leq C$ for each $j \in [t]$. Then $G_p$ undergoes a phase transition around $p=\frac{1}{d}$, which exhibits the Erd\H{o}s-R\'enyi component phenomenon.
\end{thm}

In this paper, we will investigate further the properties of the phase transition in \emph{irregular} high-dimensional product graphs. Firstly, we will give a negative answer to Question \ref{q:maxdeg}, showing that if the maximum degree of the base graphs is allowed to grow (as a function in $t$), then the largest component may have sublinear order for any $p = \Theta\left(\frac{1}{d}\right)$. Let us write $S(r,s)$ for the graph formed by taking the complete graph $K_{r}$ on $r$ vertices and adding $s$ leaves to each vertex of $K_r$.

\begin{thm}\label{unbounded degree} 
Let $r=r(t)$ and $s=s(t)$ be integers, which may tend to infinity as $t$ tends to infinity, such that $r = \omega(s t)$. Let $G^{(j)}=K_2$ for $1\leq j <t$, let $G^{(t)} = S(r,s)$ and let $G=\square_{i=1}^t G^{(i)}$, where we note that $d = (1+o(1))\frac{r}{s}$. Let $p\le\frac{1}{4st}$. Then, \textbf{whp} the largest component of $G_p$ has order at most $\frac{2|G|}{s}$.
\end{thm}
Note that for the base graphs defined in Theorem \ref{unbounded degree}, we have $i\left(G^{(j)}\right)=1$ for $1\le j \le t-1$, and $i\left(G^{(t)}\right)=\Omega\left(\frac{1}{s}\right)$, and so, as long as $s$ is not too large, this graph will satisfy the requirements of Theorem \ref{product graphs} regarding the isoperimetric constant. However, by choosing $s = \omega(1)$ and ${r = \omega( s^2t)}$, so that the upper bound on the edge probability $p$ in Theorem \ref{unbounded degree} is such that ${\frac{1}{4st} = \omega\left(\frac{s}{r}\right) = \omega\left( \frac{1}{d}\right)}$, we see that even significantly above the point $p_* = \frac{1}{d}$, the largest component in the percolated product graph will typically have size $o(|G|)$. The key observation in the construction is that most vertices in $G$ have degree $t$ while $d\gg t$. Thus, these vertices are likely to be isolated for $p$ which is around $\frac{1}{d}$.

However, we are able to give a positive answer to Question \ref{q:expdecay}, showing that in fact a threshold for the existence of a linear sized component in a percolated product graph exists even when the isoperimetric constants of the base graphs are \emph{super-exponentially} small in $t$. Furthermore, we strengthen Lichev's \cite{L22} result by determining the asymptotic order of the giant, and showing that it is in fact unique.
\begin{thm} \label{isoperimetric}
Let $C>0$ be a constant, and let $\epsilon>0$ be a small enough constant. Let $G^{(1)},\ldots, G^{(t)}$ be graphs such that for all $j\in[t]$, $1\le\Delta\left(G^{(j)}\right)\le C$ and $i\left(G^{(j)}\right)\ge t^{-t^{\frac{1}{4}}}$. Let $G=\square_{j=1}^{t}G^{(j)}$, let $d$ be the average degree of $G$ and let $p=\frac{1+\epsilon}{d}$. Then, \textbf{whp} 
\begin{enumerate}[(a)]
	\item $G_p$ contains a unique giant component of order $(1+o(1))y(\epsilon)|G|$, where $y(\epsilon)$ is defined according to \eqref{survival prob};
	\item  all other components of $G_p$ are of order $o(|G|)$.
	\end{enumerate}
\end{thm}
Note that, this is the same asymptotic fraction of the vertices as in Theorems \ref{ER thm} and \ref{hypercube thm}.

Whilst Theorem  \ref{isoperimetric} concerns the size of the largest component in the supercritical regime, which follows the Erd\H{o}s-R\'{e}nyi component phenomenon, we are also able to give an example that demonstrates that, if the base graphs are allowed to be irregular, the percolated product graph can also significantly deviate in behaviour from the Erd\H{o}s-R\'{e}nyi component phenomenon in the subcritical regime.

\begin{thm} \label{many stars}
Let $s$ be a large enough integer. Let $G^{(i)}=S(1,s)$, that is, a star with $s$ leaves, for every $1\le i \le t$, and let $G=\square_{i=1}^tG^{(i)}$, noting that $d = \frac{2st}{s+1}$. Let $p= \frac{c}{t}$ for $c>\frac{1}{3}$. Then, \textbf{whp} the largest component of $G_p$ is of order at least $|G|^{1-s^{-\frac{1}{6}}}$. 
\end{thm} 
Note that Theorem \ref{many stars} demonstrates the necessity of the assumption that the base graphs are regular in Theorem \ref{regular product graphs}, as well as the near optimality of the bound in Theorem \ref{product graphs} \ref{i:subcritical}, in terms of the size of the largest component in the subcritical regime. Indeed, since $t = \Theta(\log |G|)$ and $ \Delta\left(G^{(i)}\right) = s = \Theta(1)$ for all $i$, here we have that the largest component has order at least $\exp\left(-\Theta( t) \right) |G|$, which matches the bound in Theorem \ref{product graphs} \ref{i:subcritical}, up to the dependence on $p$ and $\max_i \{\Delta\left(G^{(i)}\right) \}$ in the leading constant. Furthermore, note that as $s$ grows, the bound on the largest component grows close to linear in $|G|$.

The structure of the paper is as follows. In Section \ref{sec prelim} we introduce some preliminary tools and lemmas we will use in the paper. In Section \ref{sec unbounded degree} we give a (short) proof for Theorem \ref{unbounded degree}. In Section \ref{sec bounded degree} we provide a general framework for showing the existence of a large component, in particular in high-dimensional product graphs. We then use this framework to prove Theorem \ref{isoperimetric} and, with an additional analysis of the structure of the $t$-fold product of stars, to prove Theorem \ref{many stars}, where the proofs of these two theorems are the most involved part of the paper. Finally, in Section \ref{s:discussion} we make some brief comments on our results and indicate directions for further research.

\section{Preliminaries}\label{sec prelim}
With respect to high-dimensional product graphs, our notation follows that of \cite{DEKK22}. Given a vertex $u = (u_1,u_2, \ldots, u_t)$ in $V(G)$ and $i \in [t]$ we call the vertex $u_i\in V\left(G^{(i)}\right)$ the \textit{$i$-th coordinate} of $u$. We note that, as is standard, we may still enumerate the vertices of a given set $M$, such as $M=\left\{v_1,\ldots, v_m\right\}$ with $v_i\in V(G)$. Whenever confusion may arise, we will clarify whether the subscript stands for enumeration of the vertices of the set, or for their coordinates.

When $G^{(i)}$ is a graph on a single vertex, that is, $G^{(i)}=\left(\{u\},\varnothing\right)$, we call it \textit{trivial} (and \textit{non-trivial}, otherwise). We define the \textit{dimension} of $G=\square_{i=1}^tG^{(i)}$ to be the number of base graphs $G^{(i)}$ of $G$ which are non-trivial.

Given $H\subseteq G=\square_{i=1}^tG^{(i)}$, we call $H$ a \textit{projection of} $G$ if $H$ can be written as $H=\square_{i=1}^tH^{(i)}$ where for every $1\le i\le t$, $H^{(i)}=G^{(i)}$ or $H^{(i)}=\{v_i\}\subseteq V\left(G^{(i)}\right)$; that is, $H$ is a projection of $G$ if it is the Cartesian product graph of base graphs $G^{(i)}$ and their trivial subgraphs. In that case, we further say that $H$ is the projection of $G$ onto the coordinates corresponding to the trivial subgraphs. For example, let $u_i\in V\left(G^{(i)}\right)$ for $1\le i\le k$, and let $H=\{u_1\}\square\cdots\square\{u_k\}\square G^{(k+1)}\square\cdots\square G^{(t)}$. In this case we say that $H$ is a projection of $G$ onto the first $k$ coordinates.

Given a graph $H$ and a vertex $v \in V(H)$, we denote by $C_v(H)$ the component of $v$ in $H$. We denote by $N_H(S)$ the external neighbourhood of a set $S$ in the graph $H$, by $E_H(A,B)$ the set of edges between $A$ and $B$ in $H$, and by $e_H(A,B)\coloneqq |E_H(A,B)|$. When the graph we refer to is obvious, we may omit the subscript. We omit rounding signs for the sake of clarity of presentation.

\subsection{The BFS algorithm} \label{BFS description}
For the proofs of our main results, we will use the Breadth First Search (BFS) algorithm. This algorithm explores the components of a graph $G$ by building a maximal spanning forest. 

The algorithm receives as input a graph $G$ and a linear ordering $\sigma$ on its vertices. The algorithm maintains three sets of vertices: 
\begin{itemize}
\item $S$, the set of vertices whose exploration is complete; 
\item $Q$, the set of vertices currently being explored, kept in a queue; and
\item $T$, the set of vertices that have not been explored yet.
\end{itemize}
The algorithm starts with $S=Q=\emptyset$ and $T=V(G)$, and ends when $Q\cup T=\emptyset$. At each step, if $Q$ is non-empty, the algorithm queries the vertices in $T$, in the order $\sigma$, to ascertain if they are neighbours in $G$ of the first vertex $v$ in $Q$. Each neighbour which is discovered is added to the back of the queue $Q$. Once all neighbours of $v$ have been discovered, we move $v$ from $Q$ to $S$. If $Q=\emptyset$, then we move the next vertex from $T$ (according to $\sigma$) into $Q$. Note that the set of edges queried during the algorithm forms a maximal spanning forest of $G$.

In order to analyse the BFS algorithm on a random subgraph $G_p$ of a graph $G$ with $m$ edges, we will utilise the \emph{principle of deferred decisions}. That is, we will take a sequence $(X_i \colon 1 \leq i \leq m)$ of i.i.d. Bernoulli$(p)$ random variables, which we will think of as representing a positive or negative answer to a query in the algorithm. When the $i$-th edge of $G$ is queried during the BFS algorithm, we will include it in $G_p$ if and only if $X_i=1$. It is clear that the forest obtained in this way has the same distribution as a forest obtained by running the BFS algorithm on $G_p$. 

\subsection{Preliminary Lemmas}\label{sec prelim_lemmas}
We will make use of two standard probabilistic bounds. The first is a typical Chernoff type tail bound on the binomial distribution (see, for example, Appendix A in \cite{AS16}).
\begin{lemma}\label{chernoff}
Let $n\in \mathbb{N}$, let $p\in [0,1]$, and let $X\sim Bin(n,p)$. Then for any positive $m$ with $m\le \frac{np}{2}$, 
\begin{align*}
    \mathbb{P}\left[|X-np|\ge m\right]\le 2\exp\left(-\frac{m^2}{3np}\right).
\end{align*}
\end{lemma}

The second is the well-known Azuma-Hoeffding concentration inequality (see, for example, Chapter 7 in \cite{AS16})
\begin{lemma}\label{azuma}
    Let $X = (X_1,X_2,\ldots, X_m)$ be a random vector with independent entries and with range $\Lambda = \prod_{i \in [m]} \Lambda_i$, and let $f:\Lambda\to\mathbb{R}$ be such that there exists $C=(C_1,\ldots, C_m) \in \mathbb{R}^m$ such that for every $x,x' \in \Lambda$ which differ only in the $j$-th coordinate,
    \begin{align*}
        |f(x)-f(x')|\le C_j.
    \end{align*}
    Then, for every $m\ge 0$,
    \begin{align*}
        \mathbb{P}\left[\big|f(X)-\mathbb{E}\left[f(X)\right]\big|\ge m\right]\le 2\exp\left(-\frac{m^2}{2\sum_{i=1}^mC_i^2}\right).
    \end{align*}
\end{lemma}

We will use the following result of Chung and Tetali \cite[Theorem 2]{CT98}, which bounds the isoperimetric constant of product graphs.

\begin{thm}[\cite{CT98}]\label{productiso} 
Let $G^{(1)},\ldots, G^{(t)}$ be non-trivial graphs and let $G = \square_{i=1}^{t}G^{(i)}$. Then
\[
\min_j \left\{i\left( G^{(j)}\right) \right\} \geq i(G) \geq \frac{1}{2}\min_j \left\{i\left( G^{(j)}\right) \right\} .
\]
\end{thm}

The following projection lemma, which is a key tool in \cite[Lemma 4.1]{DEKK22}, allows us to cover a small set of points in a product graph with a disjoint set of high-dimensional projections. This allows us to explore the neighbourhoods of these points in the percolated subgraph in an independent fashion.
\begin{lemma}[Projection Lemma]\label{projection lemma}
Let $G=\square_{i=1}^{t}G^{(i)}$ be a product graph with dimension $t$. Let $M\subseteq V(G)$ be such that $|M|=m\le t$. Then, there exist pairwise disjoint projections $H_1, \ldots, H_m$ of $G$, each having dimension at least $t-m+1$, such that every $v\in M$ is in exactly one of these projections.
\end{lemma}

We will make repeated use of the following simple, but powerful, observation of Lichev \cite{L22} about the degree distribution in a product graph. If $G$ is a product graph in which all base graphs have maximum degree bounded by $C$, then for any $v,w \in V(G)$ we have
\begin{align}\label{Lichev's observation}
    |d_G(v)-d_G(w)|\le (C-1)\cdot dist_G(v,w).
\end{align}
In particular, if $v$ is a vertex with degree significantly above or below $d$, then, despite the fact that $G$ may be quite irregular, vertices close to $v$ will still have degree significantly above or below $d$, and the percolated subgraph close to $v$ will look either super- or sub-critical, respectively.

Finally, for the proof of Theorem \ref{many stars}, we will utilise the structure, and in particular the isoperimetric inequalities, of both the \textit{Hamming graph} and the \textit{Johnson graph}.

Given positive integers $t$ and $s$, the  Hamming graph $H(t,s)$ is the graph with vertex set $[s]^t$ in which two vertices are adjacent if they differ in a single coordinate. Alternatively, $H(t,s)$ can be defined as the $t$-fold Cartesian product of the complete graphs $K_s$. In particular, it follows from Theorem \ref{productiso} that for any non-negative integers $z\le t$,
\begin{equation}\label{e:Hammingiso}
i(H(t-z,s)) \geq \frac{1}{2} i(K_s) = \frac{s}{2}.
\end{equation}

Given positive integers $t$ and $z$, the Johnson graph $J(t,z)$ is the graph with vertex set $\binom{[t]}{z}$ in which two $z$-sets $I$ and $K$ are adjacent if $|I \triangle K| = 2$. We note that $J(t,z)$ is the induced subgraph of the square of $Q^t$ on the vertex set of the $z$-th \emph{layer}, $\binom{[t]}{z}$. We will make use of the next vertex-isoperimetric inequality for Johnson graphs, which follows from the work of Christofides, Ellis and Keevash \cite{CEK13}.
\begin{thm}[\cite{CEK13}] \label{j iso} 
Let $t$ and $z$ be positive integers with $z< t$ and let $\alpha\in (0,1)$. Then there exists a constant $c>0$ such that for any subset $A \subseteq V(J(t,z))$ of size $|A| = \alpha \binom{t}{z}$
\[
\left|N_{J(t,z)}(A) \right|\geq c \sqrt{\frac{t}{z(t-z)}} (1-\alpha) |A|.
\]
\end{thm}

\section{Unbounded degree}\label{sec unbounded degree}
Given $r=r(t)$ and $s=s(t)$, recall that $S(r,s)$ is the graph formed by taking a complete graph $K_r$ on $r$ vertices and adding $s$ leaves to each vertex of $K_r$. Note that $i(S(r,s)) = \Omega\left(\frac{1}{s}\right)$. Indeed, let $A\subseteq V(S(r,s))$ with $|A|\le \frac{r(s+1)}{2}$. Let $A_1=A\cap V(K_r)$ and let $A_2=A\setminus A_1$. First, suppose that $\frac{3r}{4}\le |A_1|\le r$. Then at least, say, $\frac{|A_1|}{10}$ of the vertices of $A_1$ have at least $\frac{s}{10}$ leaves outside of $A$, and thus $|N(A)|=\Theta(rs)=\Omega(|A|)$. Now, if $0<|A_1|<\frac{3r}{4}$, then $|N(A)|\ge r-|A_1|=\Omega\left(\frac{|A|}{s}\right)$. Finally if $A_1=\varnothing$, then clearly $|N(A)|\ge \frac{|A|}{s}$. In either case, we have that $|N(A)|=\Omega\left(\frac{|A|}{s}\right)$, and thus $i(S(r,s))=\Omega\left(\frac{1}{s}\right)$.

Let $G^{(j)}=K_2$ for $1 \leq j < t$, let $G^{(t)} = S(r,s)$ and let $G=\square_{i=1}^tG^{(i)}$. Observe that
\[
|V(G)| = 2^{t-1}|S(r,s)| = 2^{t-1}(r+rs) = 2^{t-1}r(s+1),
\]
and, as long as $r =\omega( st)$,
\begin{align*}
    |E(G)| &= \frac{1}{2}2^{t-1}r\left(s+r-1+(t-1)\right)+\frac{1}{2}2^{t-1}rs\left(1+(t-1)\right)
\\&=\frac{1}{2} 2^{t-1} \bigg(r (r+s+t-2) + rs \left(1+(t-1)\right) \bigg) = (1+o(1)) 2^{t-2}r^2,    
\end{align*}
and so, if $s = \omega(1)$, we have that $d = (1+o(1)) \frac{r^2}{r(s+1)} = (1+o(1))\frac{r}{s}$.

In order to prove Theorem \ref{unbounded degree}, we will show that typically, when $p \leq \frac{1}{4st}$, almost all the vertices of $G_p$ are isolated vertices.
\begin{lemma} \label{isolated vertices}
If $p \leq \frac{1}{4st}$, then \textbf{whp} at least $2^{t-1}rs \left(1 - \frac{1}{s} \right)$ vertices of $G_p$ are isolated vertices.
\end{lemma}
\begin{proof}
Let $\{z_i \colon i \in [r]\} \subseteq V(S(r,s))$ be the vertex set of the complete graph $K_r$ and let $\{\ell_{i,j} \colon j \in [s]\}$ be the set of leaves adjacent to the vertex $z_i$ for each $i\in [r]$. Let
\[
L = \{ x \in V(G) \colon x_t = \ell_{i,j} \text{ for some } i\in[r], j\in[s]\}.
\]
Note that $d_G(x) = t$ for every $v \in L$ and that $|L| = 2^{t-1}rs$.

Let $X_L$ be the number of edges in $G_p$ which are incident with a vertex in $L$. Clearly $X_L$ is stochastically dominated by $Bin(|L|t,p)$ and hence, by Chebyshev's inequality, \textbf{whp}
\[
X_L\le 2|L|tp\le \frac{|L|}{2s}.
\]
Therefore, \textbf{whp}, the number of isolated vertices in $L$ is at least
\[
|L|\left(1-\frac{1}{s}\right) \geq 2^{t-1}rs\left(1-\frac{1}{s}\right).
\]
\end{proof}

The proof of Theorem \ref{unbounded degree} will now be a fairly straightforward corollary of Lemma \ref{isolated vertices}.
\begin{proof}[Proof of Theorem \ref{unbounded degree}]
Since $|G| = 2^{t-1}r(s+1)$ and by Lemma \ref{isolated vertices} \textbf{whp} there are at least $2^{t-1}rs\left(1-\frac{1}{2s}\right)$ isolated vertices in $G_p$, it follows that \textbf{whp} the number of vertices in any non-trivial component is at most
\begin{align*}
2^{t-1}r(s+1) - 2^{t-1}rs \left(1-\frac{1}{s}\right) &= |G| \left( 1 -\left(1 - \frac{1}{s+1} \right)\left(1-\frac{1}{s}\right)\right) 
=\frac{2|G|}{s+1} 
\leq \frac{2|G|}{s}.
\end{align*}
\end{proof}

\section{Irregular graphs of bounded degree} \label{sec bounded degree}
In many standard proofs of the existence of a phase transition in percolation models, for example in \cite{AKS81, BKL92,DEKK22,ER59,L22} and many more, in order to show the existence of a linear sized component in the supercritical regime, one first shows that \textbf{whp} a positive proportion of the vertices in the host graph $G$ are contained in \emph{big} components, that is, components of size at least $k$ for some appropriately chosen threshold $k$. One then completes the proof by using a sprinkling argument to show that \textbf{whp} almost all of these big components merge into a single, giant component.

More concretely, the first part of the argument is normally shown by bounding from below the probability that a fixed vertex is contained in a big component, together with some concentration result. In most cases, at least on a heuristic level, this is done by comparing a component exploration process near a vertex $v$ to some supercritical branching process.

In some ways, \eqref{Lichev's observation} allows us to say that, in product graphs, the vertex degree is almost constant \textit{locally}. Hence, the threshold probability above which it is likely that $v$ will lie in a big component will depend not on the global parameters of the graph, but rather just locally on the degree of the vertex $v$.

We will use this observation in two ways. Firstly, when the base graphs have bounded maximal degree, typical concentration bounds will imply that almost all the vertices have degree very close to the average degree of the graph, and so these methods will allow us to estimate the proportion of vertices which are typically contained in big components, and hence eventually the giant component, in the proof of Theorem \ref{isoperimetric}. The second way, which is perhaps more unusual, will be to apply this reasoning to the vertices in the host graph whose degrees are significantly higher than the average degree, to see that \textbf{whp} many of these vertices are contained in big components even in the subcritical regime, which we will use in the proof of Theorem \ref{many stars}.

We thus begin with the following lemma, utilising the tools developed in \cite[Lemma 4.2]{DEKK22}, which provides a lower bound on the probability that a vertex $v$ belongs to a big component, provided that the percolation probability is significantly larger than $\frac{1}{d_G(v)}$. However, since it will be essential for us to be able to grow components of order super-polynomially large in $t$ in order to prove Theorem \ref{isoperimetric}, unlike in the proof of \cite[Lemma 4.2]{DEKK22}, we will need to run our inductive argument for $\omega(1)$ steps, which will necessitate a more careful and explicit handling of the error terms and probability bounds in the proof.

\begin{lemma}\label{t^k} 
Let $C>1$ be a constant and let $\epsilon>0$ be sufficiently small. Let $G^{(1)}, \ldots, G^{(t)}$ be 
graphs such that $1\le \Delta\left(G^{(i)}\right)\le C$ for all $i\in [t]$. Let $G=\square_{i=1}^{t}G^{(i)}$. Let $v\in V(G)$ be such that $d\coloneqq d_G(v) = \omega(1)$. Let $k=k(t) = o\left( d^{1/3}\right)$ be an integer and let $m_k(d)=d^\frac{k}{6}$. Then, for any $p\ge \frac{1+\epsilon}{d}$ there exists
$c=c(\epsilon,d,k)\ge \left(\frac{\epsilon}{5}\right)^k$ such that 
\begin{align*}
    \mathbb{P}\left[|C_v(G_p)|\ge cm_k\right]\ge y-o(1),
\end{align*}
where $y \coloneqq y(\epsilon)$ is as defined in (\ref{survival prob}). 
\end{lemma}
\begin{proof}
We argue by induction on $k$, over all possible values of $C, d$, small enough $\epsilon$ and all possible choices of $G^{(1)},\ldots, G^{(t)}$.

For $k=1$, we run the BFS algorithm (as described in Section \ref{BFS description}) on $G_p$ starting from $v$ with a slight alteration: we terminate the algorithm once $\min\left(|C_v(G_p)|,d^{\frac{1}{2}}\right)$ vertices are in $S\cup Q$. Note that at every point in the algorithm we have $|S\cup Q|\le d^{\frac{1}{2}}$. Therefore, since at each point in the modified algorithm $S \cup Q$ spans a connected set, by (\ref{Lichev's observation}), at each point in the algorithm the first vertex $u$ in the queue has degree at least $d-Cd^{\frac{1}{2}}$ in $G$, and thus has at least $d-(C+1)d^{\frac{1}{2}}$ neighbours (in $G$) in $T$. 

Hence, we can couple the tree $B_1$ built by this truncated BFS algorithm with a Galton-Watson tree $B_2$ rooted at $v$ with offspring distribution $Bin\left(d-(C+1)d^{\frac{1}{2}}, p\right)$ such that $B_1\subseteq B_2$ as long as $|B_1|\le d^{\frac{1}{2}}$. Therefore, since 
\begin{align*}
    \left(d-(C+1)d^{\frac{1}{2}}\right)\cdot p&\ge \frac{(1+\epsilon)\left(d-(C+1)d^{\frac{1}{2}}\right)}{d}\\
    &\ge (1+\epsilon)\left(1-(C+1)d^{-\frac{1}{2}}\right)\\
    &\ge1+\epsilon- 2Cd^{-\frac{1}{2}},
\end{align*}
standard results imply that $B_2$ grows infinitely large with probability at least $y\left(\epsilon-2Cd^{-\frac{1}{2}}\right)-o(1)$ (see, for example, \cite[Theorem 4.3.12]{D19}). Thus, by the above and by (\ref{survival prob}), with probability at least $y-o(1)$ we have that $|C_v(G_p)|\ge d^{\frac{1}{2}}$. Since $c(\epsilon,d,1) \leq 1$, it follows that the statement holds for $k=1$, for all $C, d$, small enough $\epsilon$ and all possible choices of $G^{(1)},\ldots, G^{(t)}$. 

Let $k\ge 2$ and assume the statement holds with $c(\epsilon', d, k-1)=\left(\frac{\epsilon'}{5}\right)^{k-1}$ for all $C, d$, small enough $\epsilon'$ and all possible choices of $G^{(1)}, \ldots, G^{(t)}$. We argue via a two-round exposure. Set $p_2=d^{-\frac{4}{3}}$ and $p_1=\frac{p-p_2}{1-p_2}$ so that $(1-p_1)(1-p_2)=1-p$. Note that $G_p$ has the same distribution as $G_{p_1}\cup G_{p_2}$, and that $p_1=\frac{1+\epsilon'}{d}$ with $\epsilon'\ge\epsilon-d^{-\frac{1}{3}}$. In fact, we will not expose either $G_{p_1}$ or $G_{p_2}$ all at once, but in several stages, each time considering only some subset of the edges.

We begin in a manner similar to $k=1$. We run the BFS algorithm on $G_{p_1}$, starting from $v$, and we terminate the exploration once $\min\left(|C_v(G_{p_1})|, d^{\frac{1}{2}}\right)$ vertices are in $S\cup Q$. Once again, by standard arguments, we have that $|C_v(G_{p_1})|\ge d^{\frac{1}{2}}$ with probability at least $y(\epsilon'-2Cd^{-\frac{1}{2}})-o(1)=y-o(1)$ (where the last equality is by (\ref{survival prob})).

Let $W_0\subseteq C_v(G_{p_1})$ be the set of vertices explored in this process, and let $\mathcal{A}_1$ be the event that $|W_0| = d^{\frac{1}{2}}$, whereby the above 
\begin{equation}\label{e:p1}
\mathbb{P}[\mathcal{A}_1] \geq y-o(1).
\end{equation}
We assume in what follows that $\mathcal{A}_1$ holds. Let us write $W_0=\{v_1,\ldots, v_{d^{\frac{1}{2}}}\}$, and note that $v\in W_0$ is one of the $v_i$. Using Lemma \ref{projection lemma}, we can find pairwise disjoint projections $H_1,\ldots, H_{d^{\frac{1}{2}}}$ of $G$, each having dimension at least $t-d^{\frac{1}{2}}$, such that each $v_i\in W_0$ is in exactly one of the $H_i$ (see the first and second steps in Figure \ref{fig:indc step}). Note that $d\le \Delta(G)\le Ct$, and so $d^{\frac{1}{2}} \leq t$, thus we can apply Lemma \ref{projection lemma}.  

Thus, it follows from observation (\ref{Lichev's observation}) that, for all $i\in \left[d^{\frac{1}{2}}\right]$, we have
\begin{align*}
    |d_G(v)-d_G(v_i)|\le (C-1)\cdot dist_G(v,v_i)\le Cd^\frac{1}{2},
\end{align*}
and therefore
\begin{align*}
    d_{H_i}(v_i)\ge d_G(v_i)-Cd^{\frac{1}{2}}\ge d_G(v)-2Cd^{\frac{1}{2}}.
\end{align*}
Let $W=\bigcup_{i\in \left[d^{\frac{1}{2}}\right]}N_{H_i}(v_i)$. Then, $W\subseteq N_G(W_0)$ and, since the $H_i$ are pairwise disjoint, by the above $|W|\ge d^{\frac{3}{2}}-2Cd$.

We now expose the edges between $W_0$ and $W$ in $G_{p_2}$ (see the third step in Figure \ref{fig:indc step}). Let us denote the vertices in $W$ that are connected with $W_0$ in $G_{p_2}$ by $W'$. Then, $|W'|$ stochastically dominates
\begin{align*}
    Bin\left(d^{\frac{3}{2}}-2Cd,p_2\right).
\end{align*}
Thus, if we let $\mathcal{A}_2$ be the event that $|W'|\ge \frac{d^{\frac{1}{6}}}{3}$ and $|W'|\le d^{\frac{1}{2}}$, then by Lemma \ref{chernoff} we have that 
\begin{equation}\label{e:p2}
\mathbb{P}[\mathcal{A}_2| \mathcal{A}_1]\geq 1-\exp\left(-\frac{d^{\frac{1}{6}}}{15}\right).
\end{equation}
We assume in what follows that $\mathcal{A}_2$ also holds.

Let $W_i'=W'\cap V(H_i)$, and note that the vertices in $W'$ are neighbours in $G$ of $v_i$. Now, for each $i$, we apply once again Lemma \ref{projection lemma} to find a family of $|W'_i|\coloneqq\ell_i\le d^{\frac{1}{2}}$ pairwise disjoint projections of $H_i$, which we denote by $H_{i,1},\ldots, H_{i,\ell_i}$, such that every vertex of $W'_i$ is in exactly one of the $H_{i,j}$, and each of the $H_{i,j}$ is of dimension at least $t-2d^{\frac{1}{2}}$ (see the fourth step in Figure \ref{fig:indc step}). Furthermore, we denote by $v_{i,j}$ the unique vertex of $W_i'$ that is in $H_{i,j}$ (note that, by the above, $v_{i,j}$ is a neighbour of $v_i$ in $G$). Again, by (\ref{Lichev's observation}), for all $i,j$
\begin{align*}
    d_{H_{i,j}}(v_{i,j})\ge d_G(v_{i,j})-2Cd^{\frac{1}{2}} \geq d_G(v_i) - 2C d^{\frac{1}{2}} - C \geq d_G(v)-4Cd^{\frac{1}{2}},
\end{align*}
where we used that, by \eqref{Lichev's observation}, $d_G(v_{i}) - d_G(v_{i,j}) \leq C\cdot dist_G(v_{i},v_{i,j}) = C$. 

Crucially, note that when we ran the BFS algorithm on $G_{p_1}$, we did not query any of the edges in any of the $H_{i,j}$. Indeed, we only queried edges in $W_0$ and between $W_0$ and its neighbourhood, and by construction $E(H_{i,j})\cap E\left(W_0\cup N_G(W_0)\right)=\varnothing$. Noting that 
\[
p_1\cdot d_{H_{i,j}}(v_{i,j})\ge \frac{1+\epsilon-d^{-\frac{1}{3}}}{d}\left(d-4Cd^{\frac{1}{2}}\right)\ge 1+\epsilon-2d^{-\frac{1}{3}},
\]
%(Sahar) : Here we can see that both p_2 and the dimension reduction comes into play, and whicever affects us first is the dominant factor. Indeed, the choice of p_2 affects p_1, and in particular how much smaller it is w.r.t 1+\epsilon/d. The dimension reduction comes into play in the second term, that is, in the degree of the vertex. Should it have been "worse", it would become the dominant factor here. 
%The dominant factor here is important: later, when we evaluate the constant of the induction hypothesis, it is of the form c(\epsilon-???,k-1,d), where ??? stands here for whichever has the stronger effect on us - the dimension reduction or the choice of p_2 (in the proof here, it is the latter). For the induction to work, c(\epsilon',k,d) must be a constant to the power of k. Hence, we need that k\ll 1/??? for the induction step to work. In this proof, it becomes k\ll d^{1/3}. If, for example, we had chosen to truncate our BFS at d^{3/4} instead of d^{1/2}, we would lose d^{3/4} dimension at each inductive step. The second term in the multiplication above would then be (d-4Cd^{3/4}), and we would have 1+\epsilon-5Cd^{-1/4} in the last line. Plugging it in later on, we would have to stop already at k\ll d^{1/4}.

we may thus apply the induction hypothesis to $v_{i,j}$ in $G_{p_1}\cap H_{i,j}$ and conclude that \begin{align*}
    \mathbb{P}\left[|C_{v_{i,j}}\left(G_{p_1}\cap H_{i,j}\right)|\ge c\left(\epsilon-2d^{-\frac{1}{3}},d_{H_{i,j}}(v_{i,j}), k-1\right)m_{k-1}\left(d_{H_{i,j}}(v_{i,j})\right)\right]
    &\ge y\left(\epsilon-2d^{-\frac{1}{3}}\right)-o(1)\\
    &\ge y-o(1),
\end{align*} 
where the first inequality follows from the induction hypothesis, and the second inequality follows from (\ref{survival prob}). Furthermore, these events are independent for each $H_{i,j}$ (see the fifth step in Figure \ref{fig:indc step}).

Let us define the following indicator random variables
\begin{align*}
    \mathbb{I}(v_{i,j})\coloneqq \begin{cases}
        1 &\quad \text{if } |C_{v_{i,j}}(G_{p_1} \cap H_{i,j})|\ge c\left(\epsilon-2d^{-\frac{1}{3}},d_{H_{i,j}}(v_{i,j}),k-1\right)m_{k-1}\left(d_{H_{i,j}}(v_{i,j})\right); \\
        0 &\quad \text{otherwise,}
    \end{cases}
\end{align*}
and let $\mathcal{A}_3$ be the event that
\begin{align*}
 \sum_{v_{i,j}\in W'}
 \mathbb{I}(v_{i,j})\ge \frac{y d^{\frac{1}{6}}}{4}.
\end{align*}
Then, by Lemma \ref{chernoff} and by (\ref{survival prob}),
\begin{equation}\label{e:p3}
\mathbb{P}\left[\mathcal{A}_3|\mathcal{A}_2,\mathcal{A}_1\right] \ge 1-\exp\left(-\frac{yd^{\frac{1}{6}}}{20}\right) \geq 1 - \exp\left( -\frac{\epsilon d^{\frac{1}{6}}}{20}\right).
\end{equation}

In particular, by \eqref{e:p1}, \eqref{e:p2} and \eqref{e:p3}, 
\begin{equation}\label{e:conclusion}
\mathbb{P}[\mathcal{A}_1 \cup \mathcal{A}_2 \cup \mathcal{A}_3] \geq  y-o(1)-\exp\left( -\frac{\epsilon d^{\frac{1}{6}}}{20}\right)-\exp\left(-\frac{d^{\frac{1}{6}}}{15}\right)=y-o(1).
\end{equation}

However, we note that 
\begin{align*}
    c\left(\epsilon-2d^{-\frac{1}{3}},d_{H_{i,j}}(v_{i,j}),k-1\right)m_{k-1}
    &\geq \left( 1 - \frac{1}{2\epsilon d^{\frac{1}{3}}} \right)^{k-1}\left( \frac{\epsilon}{5} \right)^{k-1}m_{k-1}\left(d_{H_{i,j}}(v_{i,j})\right)\\
    &\geq (1-o(1)) \left( \frac{\epsilon}{5} \right)^{k-1} m_{k-1}\left(d_{H_{i,j}}(v_{i,j})\right),
\end{align*}
where the last inequality follows since $k = o\left( d^{1/3}\right)$. Hence, if $\mathcal{A}_3$ holds, then by \eqref{survival prob}

\begin{align*}
    |C_v(G_p)|&\ge \left(\sum_{v_{i,j}\in W'}\mathbb{I}(v_{i,j})\right)\cdot c\left(\epsilon-2d^{-\frac{1}{3}},d_{H_{i,j}}(v_{i,j}),k-1\right)m_{k-1}\left(d_{H_{i,j}}(v_{i,j})\right)
    \\
    &\ge\left(\frac{y}{4}d^{\frac{1}{6}}\right)\cdot (1-o(1)) \left( \frac{\epsilon}{5} \right)^{k-1}(d-4Cd^{\frac{1}{2}})^{\frac{k-1}{6}}\\
    &\ge \left(\frac{\epsilon}{5}\right)^km_k,
\end{align*}
and so the induction step holds by \eqref{e:conclusion}.
\end{proof}
Figure \ref{fig:indc step} illustrates the induction step in the above Lemma. 
\begin{figure}[H]
\centering
\includegraphics[width=0.6\textwidth]{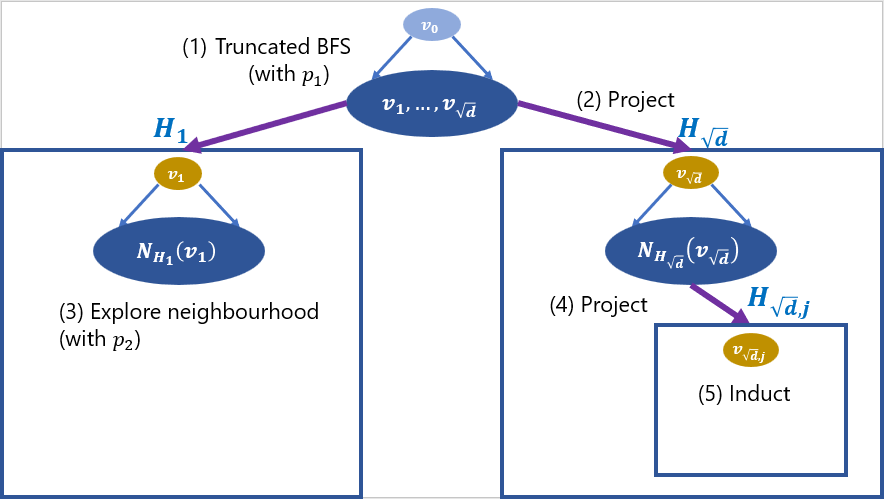}
\caption{The induction step in Lemma \ref{t^k}}
\label{fig:indc step}
\end{figure}

We note that the choice of $p_2 = d^{-\frac{4}{3}}$ in the proof of Lemma \ref{t^k} was relatively arbitrary, and we could take $p_2 = d^{-\gamma}$ for any $1 < \gamma < \frac{3}{2}$, which would lead to a similar statement for all $k = o\left( d^{\gamma -1}\right)$, with
\[
c = \left(\frac{\epsilon}{5}\right)^k \qquad \text{ and } \qquad m_k = d^{\left( \frac{3}{2} - \gamma\right)k}.
\]
In particular, this lemma can be utilised similarly for components almost as large as $d^{d^{\frac{1}{2}}}$, by taking $\gamma$ arbitrarily close to $\frac{3}{2}$ and choosing $k$ appropriately.

In $G(d+1,p)$, isoperimetric considerations alone are enough to guarantee that typically all the big components merge after a sprinkling step. In many other cases, for example in the percolated hypercube $Q^d_p$, isoperimetric considerations alone will not suffice, and a key step to proving that the big components likely merge into a giant component is to show that the vertices in the big components are in some way 'densely' spread throughout the host graph $G$. In an irregular product graph this might not be true, but the following lemma, which is a generalised version of \cite[Lemma 4.5]{DEKK22}, shows that at the very least the big components are typically well-distributed around the vertices of \textit{large degree}.

\begin{lemma}\label{density lemma} 
Let $C>1$ be a constant and let $\epsilon>0$ be a small enough constant. Let $G^{(1)},\ldots, G^{(t)}$ be graphs such that for all $i\in [t]$, $1\le \Delta\left(G^{(i)}\right)\le C$. Let $d=d(t) = \omega(1)$ and let $p \geq \frac{1+\epsilon}{d}$. Then, \textbf{whp}, there are at most $\exp(-d^{\frac{3}{2}})|G|$ vertices $v\in V(G)$ such that $d_G(v)\ge d$ and all components of $G_p$ of order at least $d^{d^{\frac{1}{3}}}$ are at distance (in $G$) greater than two from $v$. 
\end{lemma}
\begin{proof}
Let $\epsilon$ be a small enough constant. Fix $v\in V(G)$ with $d_G(v)\ge d$, and let us write $v=(v_1,\ldots, v_t)$ such that $\forall i\in[t], v_i\in V\left(G^{(i)}\right)$. Denote by $w_i$ an arbitrarily chosen neighbour of $v_i$ in $G^{(i)}$. Furthermore, for any $\ell,m$ such that $1\le \ell\neq m \le \epsilon^2d$ we define
\begin{align*}
    \begin{cases}
        H^{(i)}_{\ell,m}=\{v_i\} & \text{when } 1\le i \le \epsilon^2d, i\notin \left\{\ell, m\right\}; \\
        H^{(i)}_{\ell,m}=\{w_i\} & \text{when } i\in\left\{\ell,m\right\}; \text{ and,}\\
        H^{(i)}_{\ell,m}=G^{(i)} & \text{when } i>\epsilon^2d.
    \end{cases}
\end{align*}
%Note that, since $d \leq \Delta(G) \leq Ct$, and $\epsilon>0$ is sufficiently small, these projections are well-defined. 
Let $H_{\ell,m}=\square_{i=1}^{t}H^{(i)}_{\ell,m}$.

We defined in this manner $\binom{\epsilon^2 d}{2}\ge \frac{\epsilon^4d^2}{3}$ pairwise disjoint projections $H_{\ell,m}$, each of dimension $t - \epsilon^2d$, such that each $H_{\ell,m}$ has a vertex at distance $2$ (in $G$) from $v$, which we denote by $w_{\ell,m}$. Observe that by (\ref{Lichev's observation}), for $\epsilon>0$ sufficiently small,
\begin{align*}
    d_{H_{\ell,m}}(w_{\ell,m})\ge d_G(v)-C\epsilon^2 d-2(C-1)\ge (1-2C\epsilon^2)d\ge \left(1-\frac{\epsilon}{2}\right)d.
\end{align*}
Thus, by Lemma \ref{t^k}, with an appropriate choice of $k$, we have that $w_{\ell,m}$ belongs to a component of $\left(H_{{\ell,m}}\right)_p$ of order $d^{d^{\frac{1}{3}}}$ with probability at least $y\left( \frac{\epsilon}{3}\right)-o(1)> \frac{\epsilon}{3}$. Since the $H_{\ell,m}$ are pairwise disjoint, these events are independent for different $w_{\ell,m}$. Thus, by Lemma \ref{chernoff}, there is some $c'(\epsilon) >0$ such that at least one of the $w_{\ell,m}$ belongs to a component whose order is at least $d^{d^{\frac{1}{3}}}$ with probability at least $1-\exp(-c'd^2)$. Hence, the expected number of vertices $v\in V(G)$ such that $d_G(v)\ge d$ and $v$ is not distance at most $2$ from a component of order $d^{d^{\frac{1}{3}}}$ is at most $\exp(-c'd^2)|G|$. Therefore, by Markov's inequality, \textbf{whp} there are at most $\exp(-d^{\frac{3}{2}})|G|$ such vertices.
\end{proof}
Note, in particular, that since $d = \Theta(t)$, if the orders of the base graphs $G^{(i)}$, $|G^{(i)}|$, are growing sufficiently slowly, then $\exp(-d^{\frac{3}{2}})|G| = o(1)$, and so \textbf{whp} \emph{all} the vertices of $G$ of large order will be close to a big component.

Such a `density lemma' can then be used in graphs with sufficiently good isoperimetric inequalities to show that the big components in the percolated subgraph are typically so well connected that, after a sprinkling step, they all merge \textbf{whp}. More concretely,
one often shows that between any suitable partition of the big components, the host graph will typically contain a large family of short edge-disjoint paths between the two sides of this partition. If the number and length of these paths is sufficiently large/small, respectively, compared to the size of the big components, we can conclude that after sprinkling \textbf{whp} a large proportion of these components merge. Since we will use a variant of this argument in the proof of both Theorems \ref{isoperimetric} and \ref{many stars}, let us state a very general version of it.

Given graphs $H\subseteq G$ and a subset $X \subseteq V(G)$, we say a partition $X = A \cup B$ is \emph{($H$-)component respecting} if for every component $K$ of $H$, $K\cap X$ is fully contained in either $A$ or $B$.

\begin{lemma}\label{l:sprinkling}
Let $p \in (0,1)$, let $c \in \left(0,\frac{1}{2}\right)$ and let $m\in \mathbb{N}$. Assume $k=\omega\left(rp^{-m}\right)$. Let $H\subseteq G$ be graphs and let $X \subseteq V(G)$ be a subset such that 
\begin{enumerate}
    \item[(C1)]\label{i:order} for every component $K$ of $H$ which meets $X$,  we have that $|X\cap K|\ge k$;
    \item[(C2)]\label{i:paths} for any partition $X=A\cup B$ with $|A|, |B|\ge c|X|$ that is $H$-component respecting there is a family of at least $\frac{|X|}{r}$ edge-disjoint $A-B$ paths of length at most $m$ in $G$.
\end{enumerate}
Then \textbf{whp} $H\cup G_p$ contains a component with at least $(1-c)|X|$ vertices of $X$.
\end{lemma}
\begin{proof}
If there is no component in $H \cup G_p$ which contains at least $(1-c)|X|$ vertices of $X$, then there is some $H$-component respecting partition $X = A \cup B$ with $|A|,|B| \geq c|X|$ such that there is no path between $A$ and $B$ in $G_p$.

By (C2), for each such $H$-component respecting partition $X=A\cup B$ there is a family of at least $\frac{|X|}{r}$ many edge-disjoint $A-B$ paths of length at most $m$ in $G$. The probability that none of these paths are in $G_p$ is at most
\[
(1-p^m)^{\frac{|X|}{r}} \leq \exp\left( - \frac{|X|p^m}{r}\right).
\]

On the other hand, by (C1), for every component $K$ of $H$ which meets $X$, we have that $|X\cap K|\ge k$. Hence, the total number of $H$-component respecting partitions $X=A\cup B$ with $|A|, |B|\ge c|X|$ is at most $2^{\frac{|X|}{k}}$.

Therefore, by the union bound, the probability that $H \cup G_p$ does not contain a component of order at least $(1-c)|X|$ is at most
\[
2^{\frac{|X|}{k}}\exp\left( - \frac{|X|p^m}{r}\right) = \exp\left(-\Omega\left(\frac{k}{rp^{-m}}\right)\right) =o(1),
\]
where the first equality is because $|X|\ge |X\cap K| \ge k$ by (C1), and the second inequality follows from our assumption that $k=\omega\left(rp^{-m}\right)$.
\end{proof}

In the remainder of this section, we prove Theorems \ref{isoperimetric} and \ref{many stars}.

\subsection{Proof of Theorem \ref{isoperimetric}}
\begin{proof}[Proof of Theorem \ref{isoperimetric}]
Note that, since all of the $G^{(i)}$ are non-trivial, we have $t \leq d \leq Ct$ and $|G| \geq 2^t$, and therefore, 
\[
t^{t^{\frac{1}{4}}} = o\left( d^{d^{\frac{3}{10}}}\right) = o(|G|).
\]

We argue via a two-round exposure. Set $p_2=\frac{1}{d\log d}$ and $p_1=\frac{p-p_2}{1-p_2}$, so that $(1-p_1)(1-p_2)=1-p$. Note that $G_p$ has the same distribution as $G_{p_1}\cup G_{p_2}$ and that $p_1=\frac{1+\epsilon-o(1)}{d}$.

Set $I\coloneqq \left[\left(1-\frac{1}{\log d}\right)d, \left(1+\frac{1}{\log d}\right)d\right]$, and denote the set of vertices whose degree in $G$ lies in the interval $I$ by $V_1\subseteq V(G)$. Note that the degree of a uniformly chosen vertex in $G$ is distributed as the sum of $t$ random variables, each of which is bounded by $C$. Since the expected degree of a uniformly chosen vertex is the average degree $d$, we have by Lemma \ref{azuma} that 
$|V\setminus V_1|\le 4\exp\left(-\frac{d}{\log^2d}\right)|G|.$

Let $W$ be the set of vertices belonging to components of order at least $d^{d^{\frac{3}{10}}}$ in $G_{p_1}$. Since for every $v\in V_1$ we have that $d_G(v)\cdot p_1=1+\epsilon-o(1)$, we have by Lemma \ref{t^k} that
\begin{align} \label{e: prob at least}
    \text{every } v\in V_1 \text{ belongs to } W \text{ with probability at least } y(\epsilon)-o(1).
\end{align}
We assume henceforth that the above stated \textbf{whp} statements of $G_{p_1}$ hold. Note furthermore that by adding or deleting an edge from $G_{p_1}$ we can change the number of vertices in $W$ by at most $2d^{d^\frac{3}{10}} = o\left( \frac{|G|}{d}\right)$. Hence, by Lemma \ref{azuma} applied to the edge-exposure martingale on $G_{p_1}$ of length $\frac{|G| d}{2}$, we see that \textbf{whp} $|W| = \left(y(\epsilon)+o(1)\right)|G|$. Note that, by Lemma \ref{density lemma}, \textbf{whp} there are at most $\exp\left(-d^{\frac{3}{2}}\right)|G|$ vertices at distance (in $G$) greater than $2$ from $W$.

Let $A\cup B$ be a partition of $W$ into two parts, $A$ and $B$, with $|A|\le |B|$ and $|A|\ge \frac{|W|}{t}$. Let $A'$ be the set of vertices in $A$ together with all vertices in $V(G)\setminus B$ at distance at most $2$ from $A$, and let $B'$ be the set of vertices in $B$ together with all vertices in $V(G)\setminus A'$ at distance at most $2$ from $A$. By Lemma \ref{isoperimetric}, we have 
$$i(G)\ge \frac{1}{2}\min_{j\in[t]}i\left(G^{(j)}\right)\ge \frac{1}{2}t^{-t^{\frac{1}{4}}}.$$  
Let $E'\coloneqq E(A', (A')^C).$
We thus have that
\begin{align}\label{e:fsize}
    |E'| =e(A', (A')^C) \ge \frac{|W|}{t} \cdot \frac{1}{2}t^{-t^{\frac{1}{4}}}
     =      \frac{|W|}{2}t^{-t^{\frac{1}{4}}-1}
     =\Omega\left(\exp\left(-t^{\frac{1}{4}}\log t\right)|G|\right).
\end{align}
Recall that \textbf{whp} $|V(G) \setminus V_1|\le 4\exp\left(-\frac{d}{\log^2d}\right)|G|$ and $|V(G) \setminus N_G^2(W)| \leq \exp\left(-d^{\frac{3}{2}}\right)|G|$, where $N_G^2(W)$ is the set of vertices at distance at most two from $W$. Thus, \textbf{whp} 
\[
|V(G) \setminus \left(V_1\cap \left(A'\cup B'\right)\right)| \leq 5\exp\left(-\frac{d}{\log^2d}\right)|G|.
\]
Therefore, \textbf{whp} the number of edges in $E'$ which do not meet $B'$ is at most
\begin{align*}
    2\Delta(G)\cdot 5\exp\left(-\frac{d}{\log^2d}\right)|G|\le 10Ct\exp\left(-\frac{d}{\log^2d}\right)|G|=o(|E'|).
\end{align*}
Hence, \textbf{whp} at least $\frac{|E'|}{2}$ of the edges in $E'$ have their end-vertices in $B'$. 

Since each vertex in $A'$ is at distance at most $2$ from $A$, and similarly every vertex in $B'$ is at distance at most $2$ from $B$, we can extend these edges to a family of $\frac{|E'|}{2}$ paths of length at most $5$ between $A$ and $B$. Since every edge participates in at most $5\Delta(G)^4$ paths of length $5$, we can greedily thin this family to a set of \[
\frac{|E'|}{50\Delta(G)^4} \geq {|W|}{t^{-t^{\frac{1}{4}} - 6}}
\]
edge-disjoint paths of length at most $5$, where the inequality follows from  \eqref{e:fsize} and the fact that $\Delta(G) \leq Ct$.

We can thus apply Lemma \ref{l:sprinkling}, with $H=G_{p_1}, c=\frac{1}{t}, X=W, p=p_2, k=d^{d^{\frac{3}{10}}},r=t^{t^{\frac{1}{4}}+6}$, and $m=5$. Indeed, by our assumptions on $G_{p_1}$, we have that \textbf{whp} both conditions in Lemma $\ref{l:sprinkling}$ hold, and
\begin{align*}
rp^{-m} &=t^{t^{\frac{1}{4}}+6} p_2^{-5}
\leq t^{t^{\frac{1}{4}}+6}\left(d\log d\right)^5
= o\left(d^{d^{\frac{3}{10}}}\right) = o(k).
\end{align*}
Hence, \textbf{whp} $G_p$ contains a component $L_1$ with at least $\left(1-\frac{1}{t}\right)|W|= (y(\epsilon)-o(1))|G|$ vertices from $W$. 

As for the remaining components, observe that by \eqref{Lichev's observation} we can couple the first $\sqrt{d}$ steps of a BFS exploration process in $G_p$ starting from $v\in V_1$ from above by a Galton-Watson branching process with offspring distribution $Bin\left(\left(1+\frac{1}{\log d}\right)d + C\sqrt{d},p\right)$. Then, by standard results on Galton-Watson trees (see, for example, \cite[Theorem 4.3.12]{D19}), we have that
\begin{align} \label{e: prob at most}
    \forall v\in V_1\colon \mathbb{P}\left[|C_v(G_p)|\ge \sqrt{d}\right]\le y(\epsilon)+o(1).
\end{align}
Let $W'$ be the set of vertices of $V_1$ which lie in components of order larger than $\sqrt{d}$ in $G_p$, noting that $W \subseteq W'$. 

Then, by (\ref{e: prob at most}), $\mathbb{E}|W'| \leq \left(y(\epsilon)+o(1)\right)|V_1|$. Furthermore, by (\ref{e: prob at least}), $\mathbb{E}(|W'|) \geq \mathbb{E}(|W|) \geq \left(y(\epsilon)-o(1)\right)|V_1|$, and so  $\mathbb{E}|W'|=\left(y(\epsilon)+o(1)\right)|V_1|$. Hence, by a similar argument as before, by Lemma \ref{azuma} we have that \textbf{whp} $|W'|= (y(\epsilon)+o(1))|V_1|$. 

In particular, every component of $G_p$ apart from $L_1$ either meets $V_1 \setminus W'$, and so has order at most $\sqrt{d} = o(|G|)$, or is contained in $(V(G) \setminus V_1) \cup (W' \setminus L_1)$, and so has order at most \[
|V(G) \setminus V_1| + |W' \setminus L_1| = o(|G|),
\]
as required.
\end{proof}

\subsection{Proof of Theorem \ref{many stars}} 

Throughout the section, we let $G$ be the $t$-fold Cartesian product of the star $S(1,s)$ with $s$ leaves. Observe that the vertex degrees in $G$ are not very well-distributed, in the following sense: whilst it is easy to see that the average degree $d\coloneqq d(G)=\frac{2st}{s+1} \approx 2t$ for large enough $s$, the number of vertices $v$ such that $d_G(v)\ge \frac{1+\epsilon}{p}$ is polynomially large in $|G|$ for any choice of $\frac{1}{3t}<p \le \frac{1-\epsilon}{d}$, and is in fact of order $|G|^{1-o_s(1)}$. In particular, for such values of $p$, even though we are in the subcritical regime, for many of the vertices the percolated graph $G_p$ looks \textit{locally supercritical}. We note that in order to make the calculations and explanations cleaner and more transparent, we will in fact restrict our attention to a smaller set of vertices, whose degree in $G$ is much larger than $(1+\epsilon)3t$.
To give some intuition to our statements and proof, observe that as with the hypercube $Q^t$, we can think of $G = \left(S(1,s)\right)^t$ as consisting of a number of \emph{layers}, $M_0,M_1,\ldots, M_t$, where the $z$-th layer, $M_z$, consists of the set of vertices in $G$ such that exactly $z$ of their coordinates are centres of stars, and so their other $t-z$ coordinates are leaves (indeed, comparing to the $t$-dimensional hypercube, one can have the $z$-th layer there as the layer with $z$ coordinates being one and the other $t-z$ coordinates being zero). It is easy to see then that the degrees of the vertices in each fixed layer are the same; explicitly, 
\[
d_G(v) = zs + (t-z) \qquad \text{ for all } v\in M_z,
\]
and that the edges of $G$ are only between `adjacent' layers, that is, between $M_z$ and $M_{z-1}$, and between $M_z$ and $M_{z+1}$. Let $\Gamma_z$ be the induced subgraph of $G$ between the layers $M_z$ and $M_{z-1}$. Observe that this graph is bipartite and biregular, such that any vertex in the $z$-th layer, $u\in M_z$, has degree $d_{\Gamma_z}(u) = zs$, and any vertex in the $(z-1)$-th layer, $v \in M_{z-1}$ has degree $d_{\Gamma_z}(v)=t-z+1$. 

Whilst these graphs $\Gamma_z$ are not product graphs, they are subgraphs of the product graph $G$, and so many of the tools we developed earlier in the paper can be adjusted and applied here. In particular, for a fixed subcritical $p > \frac{1}{3t}$, for large enough $r$, the percolation process around the vertices in $M_z$ looks locally supercritical. In fact, this phenomenon is so pronounced, that even if we restrict ourselves to the subgraph $\Gamma_z$ in which the degrees of the vertices in $M_{z-1}$ are \emph{much} smaller than in $G$, we still expect the percolation clusters around the vertices in $M_z$ to grow quite large, and in particular to contain \emph{many} vertices in $M_z$.

Explicitly, taking $z$ to be, say, $z=\frac{t}{\sqrt{s}}$, we can see that the number of paths of length two (in $\Gamma_z$), starting from a fixed $v\in M_z$ and ending in $M_z$ is at least
\[
\frac{s t}{\sqrt{s}} \left(1- \frac{1}{\sqrt{s}}\right)t = \left(1-o_s(1)\right)\sqrt{s}t^2,
\]
and so, naively, if we look in the second neighbourhood of $v$ in $(\Gamma_z)_p$, we expect to find $\Theta\left( p^2\sqrt{s}t^2 \right) = \Omega(\sqrt{s}) > 1$ vertices in $M_z$, for large enough $s$. Hence, we expect the early stages of a `two-step' exploration process in $(\Gamma_z)_p$, to stochastically dominate a supercritical branching process, and so to grow to a large size, say $\sqrt{t}$, with probability bounded away from zero.

Then, using the fact that $\Gamma_z$ is contained in a product graph $G$, we can refine the proof of Lemma \ref{t^k} and use projections in order to apply this argument inductively to build a large connected set, of polynomial size in $t$, which contains many vertices in $M_z$.

Let us now formalise the above heuristics. 
\begin{lemma}\label{two step t^k}
Let $s$ be a large enough integer, let $\epsilon, \delta>0$ be sufficiently small constants, and let $k$ be an integer. Let $\frac{1-\epsilon}{3t}\le p \le \frac{1}{t}$. Let $z$ be such that $\frac{1+\delta}{10\sqrt{s}p} \le z \le \frac{t}{\sqrt{s}}$. 
Let $M_{z}$ be the set of vertices with $z$ centre coordinates in $G$, and let $v\in M_{z}$. Then, there exist positive constants $c_1=c(z,k,\epsilon, \delta)$ and $c_2=(z,k,\epsilon, \delta)$ such that with probability at least $c_1$, there is a set of vertices $W$ such that $v\in W$, $W$ is connected in $G_p$ and $|W\cap M_{z}|\ge c_2t^{\frac{k}{6}}$.
\end{lemma}
Before proving the lemma, let us note that $p=\frac{1}{3t}$ and $z=\frac{t}{\sqrt{s}}$, for $s$ large enough, satisfy the requirements of the lemma.
\begin{proof}[Proof of Lemma \ref{two step t^k}]
Throughout the proof, we will use ideas similar to those of Lemma \ref{t^k}, while refining them to our particular setting. Let $M_{z-1}$ be the set of vertices with $z-1$ centre coordinates in $G$. Let $\Gamma\coloneqq \Gamma_{z}$ be the bipartite biregular induced subgraph of $G$ between the layers $M_{z}$ and $M_{z-1}$. Throughout the proof, we will in fact consider the percolated graph $\Gamma_p\subseteq G_p$.

We prove by induction on $k$, over all relevant values of $z$, $\epsilon$ and $\delta$.

For $k=1$, let us first describe the following variant of the BFS algorithm. Here, we maintain the following sets of vertices:
\begin{itemize}
    \item $Q$, the vertices are currently exploring, kept in a queue;
	\item $S_1 \cup S_2$, the vertices already processed in the exploration process; 
	\item $T_1 \cup T_2$, the unvisited vertices.
\end{itemize}
We begin with $Q=\{v\}$, $S_1, S_2=\emptyset$ and $T_1=M_z\setminus\{v\}$, $T_2=M_{z-1}$. At each iteration, we proceed as follows. We consider the first vertex $u\in Q$, and explore its neighbourhood in $T_2$ in $G_p$, that is, $N_{G_p}(u, T_2)$. For each vertex $x\in N_{G_p}(u, T_2)$ in turn, we do the following:
\begin{itemize}
    \item[-] We move $x$ to $S_2$;
    \item[-] We expose the neighbours of $x$ in $T_1$  in $G_p$, that is, $N_{G_p}(x,T_1)$;
    \item[-] We move the vertices  $y\in N_{G_p}(x,T_1)$ from $T_1$ to $Q$ one by one.
\end{itemize}
Finally, after exploring all relevant vertices $x\in N_{G_p}(u, T_2)$, we move $u$ from $Q$ to $S_1$. See Figure \ref{fig:indc step 2} for an illustration. Note that, by taking an edge from the vertex $u$ to each vertex $x \in N_{G_p}(u, T_2)$ and an edge from each vertex $x \in N_{G_p}(u, T_2)$ to each vertex $y\in N_{G_p}(x,T_1)$ discovered in this way, we can think of this process as building a tree $B_1 \subseteq G_p$ which spans $S_1 \cup S_2 \cup Q$ at each stage of the process. 

\begin{figure}[H]
\centering
\includegraphics[width=0.6\textwidth]{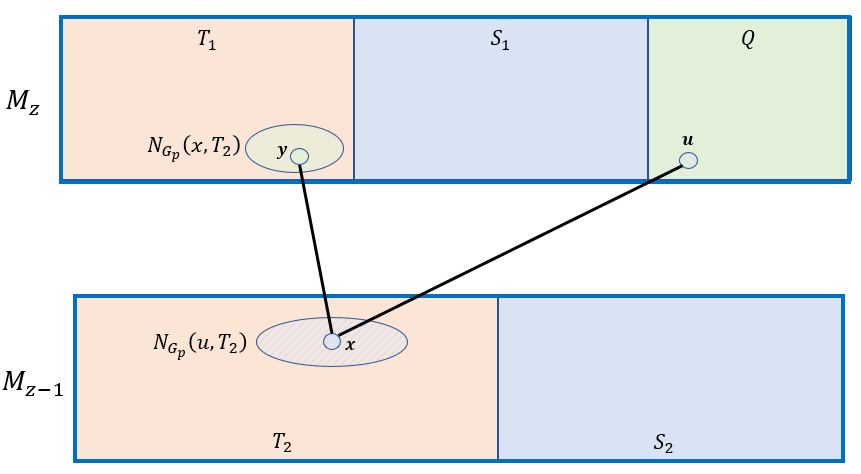}
\caption{The vertex $u\in Q$ and its exposed neighbourhood in $T_2$, and a vertex $x\in T_2$ and its exposed neighbourhood in $T_1$. Note that in the course of the exploration process, all the vertices in $N_{G_p}(x,T_2)$ move from $T_1$ to $Q$, the vertex $x$ (and, subsequently, all the vertices from $N_{G_p}(u,T_2)$) move to $S_2$, and, finally, the vertex $u$ moves to $S_1$.}
\label{fig:indc step 2}
\end{figure}

We run the algorithm until either $Q$ is empty, or $|S_1\cup S_2\cup Q|=\sqrt{t}$. Let us first note that at all times, $Q\cup S_1\cup T_1\subseteq M_z$, and $S_2\cup T_2\subseteq M_{z-1}$. 

Observe that, since at all times $|S_1\cup S_2\cup Q|\le \sqrt{t}$, we have that for every $u\in Q$, $d(u, T_2)\ge zs-\sqrt{t}\ge \frac{zs}{2}$ (since $z\ge \frac{1+\delta}{10\sqrt{s}p}\ge \frac{(1+\delta)t}{10\sqrt{s}}$), and for every $w\in T_2$, $d(w, T_1)\ge t-z+1-\sqrt{t}\ge \frac{2t}{3}$ (since $z\le \frac{t}{\sqrt{s}}$ and $s$ is large enough). We can thus couple the tree $B_1$ constructed in our truncated algorithm with a Galton-Watson tree $B_2$ rooted at $v$, such that for all $i$, the offspring distribution at the $2i$-th generation (where we start at generation $0$, with $v$) is $Bin\left(\frac{zs}{2},p\right)$, and at the $(2i+1)$-th generation is $Bin\left(\frac{2t}{3},p\right)$, such that $B_2$ is isomorphic to a subgraph of $B_1$ as long as $|B_1|\le \sqrt{t}$. 

Let us denote by $Z_i$ be the \textit{number} of vertices at depth $i$ in $B_2$ (where $Z_0=1$, and the vertex in the $0$-th depth is $v$). We then have by the above and by Lemma \ref{chernoff} that, for large enough $s$,
\begin{align*}
    \mathbb{P}\left[Z_{2(i+1)}\le \frac{zs^{\frac{2}{3}}x}{10t}\mid Z_{2i}=x\right]&\le \mathbb{P}\left[Z_{2i+1}\le \frac{zs^{\frac{2}{3}}x}{t}\mid Z_{2i}=x\right]+\mathbb{P}\left[Z_{2(i+1)}\le \frac{zs^{\frac{2}{3}}x}{10t}\mid Z_{2i+1}\ge\frac{zs^{\frac{2}{3}}x}{t}\right]\\
    &\le \mathbb{P}\left[Bin\left(\frac{zsx}{2},p\right)\le \frac{zs^{\frac{2}{3}}x}{t}\right]+\mathbb{P}\left[Bin\left(\frac{2zs^{\frac{2}{3}}x}{3},p\right)\le \frac{zs^{\frac{2}{3}}x}{10t}\right]\\
    &\le \exp\left(-2x\right)+ \exp\left(-\frac{3x}{16}\right)\le \exp\left(-\frac{x}{100}\right),
\end{align*}
where the penultimate inequality follows since $p\ge \frac{1-\epsilon}{3t}$ and $zsp\ge \frac{(1+\delta)\sqrt{s}}{10}$. Hence, we obtain:
\begin{align*}
   \mathbb{P}\left[\exists i \text{ such that } Z_{2i}\le \left(\frac{zs^{\frac{2}{3}}}{10t}\right)^i\right]&\le \sum_{i=1}^{\infty}\mathbb{P}\left[Z_{2i}\le \left(\frac{zs^{\frac{2}{3}}}{10t}\right)^i \mid Z_{2(2i-1)}\ge \left(\frac{zs^{\frac{2}{3}}}{10t}\right)^{i-1}\right]\\
   &\le \sum_{i=1}^{\infty}\exp\left(-\frac{1}{100}\cdot\left(\frac{zs^{\frac{2}{3}}}{10t}\right)^i\right)\le \exp\left(-\frac{1}{200}\right),
\end{align*}
where the last inequality follows since $p\le \frac{1}{t}$, and $z\ge \frac{(1+\delta)}{10\sqrt{s}p}\ge \frac{(1+\delta)t}{10\sqrt{s}}$, and we take $s$ to be large enough. In particular, $\frac{zs^{\frac{2}{3}}}{10t}>1$, and hence the Galton-Watson tree grows to infinity with probability at least $\exp\left(-\frac{1}{200}\right)$, and therefore with probability at least $\exp\left(-\frac{1}{200}\right)$ we have that $|B_2|=\sqrt{t}$ and hence $|S_1\cup S_2\cup Q|=\sqrt{t}$ at some point. We thus have that with the same probability, we have a connected set $W$ such that $v\in W$, $W$ is connected in $G_p$, and $|W|=\sqrt{t}$. Let us argue further that, in fact, $|W\cap \left(S_1\cup Q\right)|\ge\frac{\sqrt{t}}{22}$.
Consider any vertex $w\in S_2$. Crucially, observe that by the construction of the exploration process, unless $w$ is the last vertex we have moved such that now $|B_1|=\sqrt{t}$, after we move $w$ from $T_2$ to $S_2$ we expose the neighbourhood of $w$ in $T_1$ (see Figure \ref{fig:indc step 2}). If $|N_{G_p}(w, T_1)|\ge 1$, then we added at least one child of $w$ to $B_1$. The probability that $|N_{G_p}(w,T_1)|=0$ is at most 
\begin{align*}
    (1-p)^{\frac{2t}{3}-\sqrt{t}}\le \left(1-\frac{1-\epsilon}{3t}\right)^{\frac{t}{2}}\le \exp\left(-\frac{1-\epsilon}{6}\right)\le \frac{9}{10}.
\end{align*}
Hence the number of $w\in S_2$ which does not have at least one child, stochastically dominates $1+Bin\left(\sqrt{t},\frac{9}{10}\right)$ (where we use the fact that $|B_1|\le \sqrt{t}$ and our stochastic domination on the number of neighbours). Hence, by Lemma \ref{azuma}, the number of vertices $w\in T_2$ which do not have at least one child in $T_1$ in $B_1$ (that afterwards moves to $Q_1$ and perhaps subsequently to $S_1$) is \textbf{whp} at most $\frac{10\sqrt{t}}{11}$. Therefore, there are at most $\frac{10\sqrt{t}}{11}$ vertices in $S_2$ which do not have a child in $B_1$. Since $B_1$ is a tree, these children are distinct for distinct $w$, and they all lie in $S_1$ or $Q$. 
It follows that \textbf{whp} $|S_1\cup Q|\ge |S_2|-\frac{10\sqrt{t}}{11}$, and hence with probability at least $c>0$, we have that $|S_1\cup Q|\ge \frac{\sqrt{t}}{22}$, and hence $|W\cap M_{z}|\ge\frac{\sqrt{t}}{22}$, completing the case of $k=1$. 

We proceed in a manner somewhat similar to the proof of Lemma \ref{t^k}. Let $k\ge 2$ and assume the statement holds with $c_1(z', k-1,\epsilon',\delta')$ and $c_2(z', k-1,\epsilon',\delta')$ for all relevant values of $z'$, $\epsilon'$ and $\delta'$. We argue via a two-stage exposure, with $p_2=\frac{1}{t\log t}$ and $p_1=\frac{p-p_2}{1-p_2}$ so that $(1-p_1)(1-p_2)=1-p$. Note that $G_p$ has the same distribution as $G_{p_1}\cup G_{p_2}$, and that $p_1\ge(1-o_t(1))p\ge \frac{1-\epsilon_1}{3t}$ for some $\epsilon_1>0$ constant. Similarly, $z\ge \frac{1+\delta}{10\sqrt{s}p}\ge \frac{1+\delta_1}{10\sqrt{s}p_1}$, and we can thus begin in the same manner as $k=1$, and find with probability at least $c>0$, a set $W_0$ such that $v\in W_0$, $W_0$ is connected in $G_p$, $|W_0|=\sqrt{t}$
and $|W_0\cap M_z|\ge \frac{\sqrt{t}}{22}$. We further note that $W_0\subseteq V(\Gamma)$. We proceed under the assumption that indeed $|W_0|=\sqrt{t}$. Let us enumerate the vertices in $W_0\cap M_z$ as $\left\{v_1,\ldots,v_{\ell}\right\}$, where $\ell\ge\frac{\sqrt{t}}{22}$ and $\ell\le \sqrt{t}$ (since $|W_0|=\sqrt{t}$).

Using Lemma \ref{projection lemma}, we can find pairwise disjoint projections $H_1,\ldots, H_{\ell}$ of $G$, each having dimension at least $t-\sqrt{t}$, such that each $v_i\in W_0$ is in exactly one of $H_i$'s. 

We now intend to find pairwise disjoint projection of $G$, each with a vertex in $M_{z}$, to which we can apply our induction hypothesis. Note that, in an arbitrary projection $H'$ of $G$ with dimension $t'$, the set $M_z \cap V(H')$ will still be a layer of $H'$. However, for the vertices in this layer, the number of `active' coordinates in $H'$ which are centres of stars will be some number between $z$ and $z-(t-t')$. Our approach here is then almost identical to that of the proof of Lemma \ref{t^k}, with the slight difference that we want our vertices to be in $M_{z}$, and have $z'$ centre coordinates inside the projected subgraph, for some $z'$ such that $z'\ge \frac{1+\delta'}{10\sqrt{s}p_1}$ for some $\delta'>0$.

In order to obtain that, we begin by exposing some carefully chosen sets of edges with probability $p_2$. Observe that for each $v_i\in W_0\cap M_{z}$, we have that 
\[
|N_{H_i\cap \Gamma}(v_i)|\ge |N_\Gamma(v_i)|-s\cdot \sqrt{t}\ge\frac{zs}{3},
\]
where the first inequality follows since the dimension of $H_i$ is at least $t-\sqrt{t}$ and the maximum degree of the base graphs is $s$. We further note that $N_{H_i\cap \Gamma}(v_i)\subseteq M_{z-1}$. Let \[
W_1=\bigcup_{i\in\left[\sqrt{\ell}\right]}N_{H_i\cap \Gamma}(v_i)\subseteq M_{z-1}.
\]
Then, recalling that the $H_i$'s are pairwise disjoint, we have that
\begin{align*}
    |W_1|\ge \frac{\sqrt{t}}{25}\cdot\frac{zs}{3}\ge \frac{\sqrt{t}(1+\delta)\sqrt{s}}{750p}\ge t^{\frac{3}{2}},
\end{align*}
where we used our assumption that $z\ge \frac{1+\delta}{10\sqrt{s}p}$ and that $s$ is large enough.

We now expose the edges between $W_0$ and $W_1$ in $G_{p_2}$. Let us denote the set of vertices in $W_1$ that are connected with $W_0$ in $G_{p_2}$ by $W_1'$. Then, $|W_1'|$ stochastically dominates $Bin\left(t^{\frac{3}{2}},p_2\right)$.
Thus, by Lemma \ref{chernoff}, we have that \textbf{whp} $|W_1'|\ge \frac{\sqrt{t}}{10\log t}$. We assume in what follows that this property holds.

We now expose the edges in $G_{p_2}$ between $W_1'$ and $M_{z}\setminus W_0$, recalling that by our assumption $|W_0|=\sqrt{t}$.
Note that for every $w_i\in W_1'$, we have that 
$$|N_{H_i\cap\Gamma}(w_i) \cap (M_{z}\setminus W_0)|\ge |N_\Gamma(w_1)|-2s\cdot\sqrt{t}\ge \frac{t}{3},$$
and $N_{H_i\cap \Gamma}(w_i) \cap (M_{z}\setminus W_0)
\subseteq M_{z}$. Let 
\[
W_2=\bigcup_{i\in\left[\ell\right]} N_{H_i\cap\Gamma}(w_i)\cap (M_{z}\setminus W_0)\subseteq M_{z},
\]
where we note that by the projection Lemma (lemma \ref{projection lemma}) and our construction, these neighbourhoods are disjoint. Then \textbf{whp} $|W_2|\ge \frac{t}{3}\cdot \frac{\sqrt{t}}{10\log t}.$

Let $W_2'\subseteq M_{z}$ be the set of vertices in $W_2$ that are connected with $W_1'$ in $G_{p_2}$. Then, $|W_2'|$ stochastically dominates $Bin\left(\frac{t^{\frac{3}{2}}}{10\log t},p_2\right).$
Thus, by Lemma \ref{chernoff}, we have that \textbf{whp} $|W_2'|\ge \frac{\sqrt{t}}{20\log^2t}$ and $|W_2'|\le \sqrt{t}$. We assume henceforth that these two properties hold. 

The subsequent part of the proof is almost identical to that of the proof of Lemma \ref{t^k}, where here we only need to run the inductive step constantly many times (as opposed to $\omega(1)$ times), allowing us to analyse it in a more straightforward manner. 

Indeed, let $W_{2,i}'=W_2'\cap V(H_i)\subseteq M_{z}$. We thus have now our set of vertices in $M_{z}$. Now, for each $i$, we apply once again Lemma \ref{projection lemma} to find a family of $|W_{2,i}'|\coloneqq \ell_i\le \sqrt{t}$ pairwise disjoint projections of $H_i$, which we denote by $H_{i,1},\ldots, H_{i,\ell_i}$, such that every vertex of $W_{2,i}'\subseteq M_{z}$ is in exactly one of the $H_{i,j}$, and each of the $H_{i,j}$ is of dimension at least $t-2\sqrt{t}$. We have that for all $i,j$, the number of centre coordinates of $v_{i,j}$ in $H_{i,j}$, denote it by $z_{i,j}$, is at least $z-2s\sqrt{t}$. Recall that $p_1\ge \frac{1-\epsilon_1}{3t}$, and note that 
\begin{align*}
    z_{i,j} \ge z-2s\sqrt{t}\ge \frac{1+\delta}{10\sqrt{s}p}-2s\sqrt{t} \ge \frac{1+\delta_2}{10\sqrt{s}p} \ge \frac{1+\delta_3}{10\sqrt{s}p_1}
\end{align*}
for some $\delta_2,\delta_3>0$, where we used $p\le \frac{1}{t}$ and $p_1=(1-o(1))p$. Altogether, the $v_{i,j}$ (that is, their corresponding $z_{i,j}$) together with $p_1$ satisfy the conditions of the induction hypothesis for some $\epsilon_{i,j},\delta_{i,j}>0$.

As in the proof of Lemma \ref{t^k}, note that when we ran the algorithm on $G_{p_1}$, we did not query any of the edges inside any of the $H_{i,j}$. By the above inequality we can apply the induction hypothesis to $v_{i,j}$ in $G_{p_1}\cap H_{i,j}$ and conclude that with probability at least $c_1(z_{i,j},k-1,\epsilon_{i,j},\delta_{i,j})$,
there exists a connected set $W_{i,j}$ such that $v_{i,j}$ is $W_{i,j}$, $|W_{i,j}\cap M_{z}|\ge c_2(z_{i,j},k-1,\epsilon_{i,j},\delta_{i,j})t^{\frac{k-1}{6}}$. Let $c'_1=\min_{i,j}c_1(z_{i,j},k-1,\epsilon_{i,j},\delta_{i,j})$ and $c'_2=\min_{i,j}c_2(z_{i,j},k-1,\epsilon_{i,j},\delta_{i,j})$. Noting that the above described events are independent for every $H_{i,j}$, we obtain by Lemma \ref{chernoff} that \textbf{whp}, at least $\frac{c'_1\sqrt{t}}{40\log^2t}$ of the $H_{i,j}$ have that $|W_{i,j}\cap M_{z}|\ge c'_2t^{\frac{k-1}{6}}$. Thus, with probability at least $c'_1+o(1)\coloneqq c_1$, we have found a connected set $W$ such that $v$ is in $W$, and there is some $c_2>0$
such that
\begin{align*}
    |W\cap M_{z}|\ge \frac{c_1'\sqrt{t}}{40\log^2t}\cdot c'_2t^{\frac{k-1}{6}}> c_2t^{\frac{k}{6}},
\end{align*}
completing the induction step.
\end{proof}

We now turn to show that \textbf{whp} many of the vertices with $z\coloneqq \frac{t}{\sqrt{s}}$ centre coordinates are contained in big connected sets in $G_p$.

\begin{lemma}\label{high degree in big comp finalfinalfinal}
Let $s$ be a large enough integer. Let $p\ge \frac{1}{3t}$ and let $k>0$ be an integer. Let $z\coloneqq \frac{t}{\sqrt{s}}$. Let $M_{z}$ be the set of vertices with $z$ centre coordinates in $G$ and let $W\subseteq M_{z}$ be the set of vertices in $M_{z}$ that belong to components with at least $t^k$ vertices in $M_{z}$ in $G_p$. Then, there exists a constant $\rho>0$ such that \textbf{whp}, 
\begin{align*}
    |W|\ge \rho|M_{z}|\ge |G|^{1-s^{-\frac{1}{5}}}.
\end{align*}
\end{lemma}
\begin{proof}

Observe that in order to choose a vertex in $M_{z}$ we have $\binom{t}{z}$ choices for the $z$-set of coordinates which are the centre of a star, and then we have $s^{t-z}$ choices for the leaf coordinates. It follows that
\begin{align*}
    |M_{z}|= \binom{t}{z}s^{t-z}\ge \left(\sqrt{s}\right)^{\frac{t}{\sqrt{s}}}s^{\left(1-\frac{1}{\sqrt{s}}\right)t} \ge (s+1)^{\left(1-s^{-\frac{1}{3}}\right)t} = |G|^{1-s^{-\frac{1}{3}}},
\end{align*} 
where the penultimate inequality holds for any $s$ large enough. 

By Lemma \ref{two step t^k}, there is some constant $c' >0$ such that every $v\in M_{z}$ belongs to a connected set of order $t^k$ with probability at least $c'>0$. We thus have that
\[
\mathbb{E}[|W|]\ge c'|M_{z}|\ge c'|G|^{1-s^{-\frac{1}{3}}}.
\]

Consider the edge-exposure martingale on $G_p$. It has length $|E(G)| = \frac{st}{s+1} |G| \leq 2t |G|$, and the addition or deletion of an edge can change the value of $|W|$ by at most $2t^k$. Therefore, by Lemma \ref{azuma}, we have for $\rho=\frac{c'}{2}>0$ that
\begin{align*}
    \mathbb{P}\left[|W|\le \rho|M_{z}|\right]&\le \mathbb{P}\left[|W|\le \frac{\mathbb{E}|W|}{2}\right]
    \le 2\exp\left(-\frac{|G|^{2-2s^{-\frac{1}{3}}}}{16|G|t^{2k+1}} \right)
    \le 2\exp\left(-|G|^{1-s^{-\frac{1}{4}}}\right)=o(1),
\end{align*}
as long as $s$ is large enough. It follows that \textbf{whp}
\[
|W|\ge \rho|M_{z}|\ge \frac{c'}{2}|G|^{1-s^{-\frac{1}{4}}} \geq |G|^{1-s^{-\frac{1}{5}}}.  
\]
\end{proof}

We note that while $M_z$, where $z=\frac{t}{\sqrt{s}}$, takes a polynomial fraction of the vertices, this set of vertices still only amounts to $o(|G|)$. Thus, we cannot use global isoperimetric properties of $G$ to find typically sufficiently many short paths between relevant (respecting) partitions of $M_z$, as we do in the proof of Theorem \ref{isoperimetric}. Instead, we need to look at \textit{local} isoperimetric properties of $M_z$, such as the isoperimetric property of its two-neighbourhood in $G$, in order to show that \textbf{whp} many of the large connected sets in $G_p\cap M_z$ in fact belong to the same connected component in $G_p$.

\begin{proof}[Proof of Theorem \ref{many stars}] 
Assume that all the conditions in Theorem \ref{many stars} hold.

Once again, we argue via a two-round exposure. Let $\epsilon$ be a small enough constant, such that $p_2=\frac{\epsilon}{3t}$, $p_1=\frac{1}{3t}$ and $(1-p_1)(1-p_2)=1-p$, noting that $G_{p_1} \cup G_{p_2}$ has the same distribution as $G_p$.

Let $z\coloneqq \frac{t}{\sqrt{s}}$. Let $M_{z}$ be the set of vertices with $z$ centre coordinates in $G$, and let $W\subseteq M_{z}$ be the set of vertices in $M_{z}$ belonging to a component with at least $t^{15}$ vertices in $M_{z}$ in $G_{p_1}$. Then, by Lemma \ref{high degree in big comp finalfinalfinal} with $k=15$, we have that there exists a constant $\rho>0$ such that \textbf{whp}, 
\begin{align}\label{before_sprinkle}	
  |W|\ge \rho|M_{z}|\ge |G|^{1-s^{-\frac{1}{5}}}. 
\end{align}

We note further that, since the average degree $t = \Theta(\log |G|)$, by a similar argument as in Lemma \ref{density lemma} applied to the vertices with $z$ centre coordinates in $G$, we claim that \textbf{whp} 
\begin{align}\label{distance2}	
    \text{every vertex in $M_{z}$ is at distance at most 2 from $W$.}
\end{align}
To see why, observe that given a vertex $v\in M_{z}$, we can form $t^{\frac{4}{3}}$ pairwise disjoint projections of $G$, each isomorphic to the product of $(1-o(1))t$ copies of $S(1,s)$ and each containing a vertex with $z$ centre coordinates at distance $2$ from $v$. Indeed, since the number $z$ of coordinates of $v$ which are the centre of star satisfies $t^{\frac{2}{3}} \leq z \leq t-t^{\frac{2}{3}}$, we can assume without loss of generality that the first $t^{\frac{2}{3}}$ coordinates of $v$ are the centres of stars, and the second $t^{\frac{2}{3}}$ coordinates are leaves. 

Given a pair of coordinates $1 \leq i \leq t^{\frac{2}{3}} < j \leq 2t^{\frac{2}{3}}$, let $H_{i,j}$ be the projection of $G$ to the first $2t^{\frac{2}{3}}$ coordinates where the first $2t^{\frac{2}{3}}$ coordinates agree with $v$, except that the $i$-th coordinate is some arbitrary leaf and the $j$-th coordinate is the centre of the star. 

Then, the graphs $H_{i,j}$ form a collection of $t^{\frac{4}{3}}$ pairwise disjoint projections of $G$, each isomorphic to the product of $(1-o(1))t$ copies of $S(1,s)$ and each containing a vertex with $z$ centre coordinates in $G$ at distance $2$ from $v$. 

Since $p$ will still be `supercritical' at these vertices in each $H_{i,j}$, an argument as in Lemma \ref{density lemma} will imply that the expected number of vertices in $M_{z}$ which are not at distance at most 2 from a vertex in $W$ is at most $\exp\left( - \Omega\left(t^{\frac{4}{3}} \right)\right)|G| = o(1)$, and so the claim  follows by Markov's inequality. We assume in the following that claim \eqref{distance2} holds.

Let $A\cup B$ be a $G_{p_1}$-respecting partition of $W$ into two parts, $A$ and $B$, such that $|A|\le |B|$ and $|A|\ge \frac{|W|}{3}$. Let $A'$ be the set of vertices in $A$ together with all vertices in $M_{z}\setminus B$ that are within distance $2$ to $A$, and let $B'$ be the set of vertices in $B$ together with all vertices in $M_{z}\setminus A'$ that are within distance $2$ to $B$, so that $M_{z} = A' \cup B'$. We require the following claim, whose proof we defer to the end of this proof:
\begin{claim} \label{star iso finalfinalfinal}
\textbf{Whp}, there are at least $\frac{|W|}{t^{2}}$ paths of length $2$ in $G$ between $A'$ and $B'$. 
\end{claim}

We can extend these paths to a family of $A-B$ paths of length at most $6$ (see Figure \ref{fig:extension}) and then, very crudely, since $\Delta(G)=st$, we can greedily thin this family to a set of $\frac{|W|}{36s^6t^{8}}$ edge-disjoint $A-B$ paths of length at most $6$. 

We can now apply Lemma \ref{l:sprinkling} with $H=G_{p_1}$, $c=\frac{1}{3}$, $X=W$, $p=p_2=\frac{\epsilon}{3t}=O(t^{-1})$, $k=t^{15}$, $r=36s^6t^{8}$ and $m=6$. Indeed, by our assumptions on $G_{p_1}$ and by Lemma \ref{two step t^k}, we have that \textbf{whp} both conditions in Lemma $\ref{l:sprinkling}$ hold, and
\begin{align*}
rp^{-m} &=36s^6t^{8}p_2^{-6}
=O\left(t^{14}\right)
= o\left(t^{15}\right) = o(k).
\end{align*} 
As a consequence of Lemma \ref{l:sprinkling}, we have that \textbf{whp} $G_p=G_{p_1}\cup G_{p_2}$ contains a component which contains at least $ \frac{2}{3}|W|$ vertices of $W$. Combining this with the bound $|W|\ge|G|^{1-s^{-\frac{1}{5}}}$ from \eqref{before_sprinkle}, we have that \textbf{whp} the largest component of $G_p$ is of order at least $\frac{2}{3}|G|^{1-s^{-\frac{1}{5}}}\ge |G|^{1-s^{-\frac{1}{6}}}$. 
\end{proof}
\begin{figure}[H]
\centering
\includegraphics[width=0.5\textwidth]{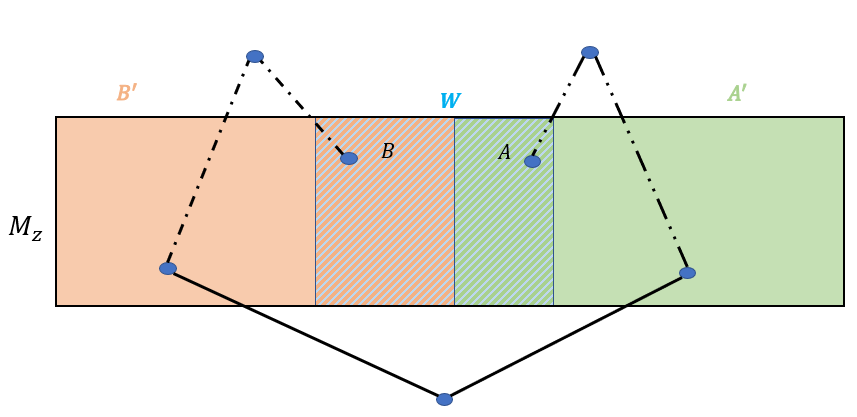}
\caption{An extension of a partition $A\cup B$ of $W$ to a partition $A'\cup B'$ of $M_{z}$. Each $A'-B'$ two-path (solid) can be extended into an $A-B$ six-path (solid and dashed). Note that some vertices in these paths may lie outside of $M_{z}$.}
\label{fig:extension}
\end{figure}

We finish this section with the proof of the above claim.
\begin{proof}[Proof of Claim \ref{star iso finalfinalfinal}]

Let $z=\frac{t}{\sqrt{s}}$ be the number of coordinates of $v$ which correspond to a centre of a star. Recall that, by Lemma \ref{high degree in big comp finalfinalfinal}, there is a constant $\rho>0$ such that \textbf{whp} $|W|\ge \rho |M_{z}|$. The following two constants will be useful throughout the proof:
\begin{align*}
    \alpha\coloneqq \frac{\rho}{4}, \quad \beta\coloneqq \frac{\rho}{3(4-\rho)}.
\end{align*}
Note that $\beta+(1-\beta)\alpha=\frac{\rho}{3(4-\rho)}+\frac{\rho}{4}\left(1-\frac{\rho}{3(4-\rho)}\right)=\frac{\rho}{3}$.

We are interested in the structure of the graph $G^2[M_{z}]$. Suppose that $v,w\in M_{z}$, and denote by $I,K\in [t]^{z}$ (respectively) the set of coordinates corresponding to centres of stars. Then $v$ and $w$ are at distance $2$ in $G$ if either 
\begin{enumerate}
    \item[(T1)] \label{i:same3}\quad $I=K$ and they disagree on a single coordinate of a leaf outside $I$; or 
    \item[(T2)] \label{i:symdiff3} \quad $|I\triangle K|=2$ and they agree on all the (leaf) coordinates outside $I\cup K$.
\end{enumerate}

Hence, we can see that $G^2[M_{z}]$ can be built in the following manner (see also Figure \ref{fig:sub1}) --- we start with the Johnson graph $J\coloneqq J(t,{z})$, which represents the ${z}$-set of coordinates which are centres of stars, and we replace each vertex with a copy of $H\coloneqq H(t-{z},s)$, which represents the set of leaf coordinates. There is a natural map $f: M_{z} \to V(J) \times V(H)$ which maps a vertex $v \in M_{z}$ to the pair $(I,x)$ where $I$ is the ${z}$-set of coordinates of $v$ which are centres of stars and $x$ is the restriction of $v$ to the coordinates which are leaves. 

The edges of $G^2[M_{z}]$ then come in two types. Those of type (T1) give rise to a copy of $H$ on $I \times V(H)$ for each $I \in V(J)$, whereas those of type (T2) give rise to a matching between $I \times V(H)$ and $K \times V(H)$ for each $(I,K) \in E(J)$, where a vertex $(I,x)$ is matched to a vertex $(K,y)$ if and only if $f^{-1}(I,x)$ and $f^{-1}(K,y)$ agree on all coordinates outside $I \cup J$. Note that this is similar to, but not quite the same as, the Cartesian product of $J$ and $H$.

We use the partition $M_{z}=A'\cup B'$ to define a colouring $\chi : V(J) \to [3]$ in the following manner. We say that a copy of $H$ given by $I \times V(H)$ is \textit{evenly-split} if
\[
\alpha |H| \leq |f^{-1}(I \times V(H)) \cap A'| \leq  \left(1-\alpha\right)|H|.
\]
Note that, since $M_{z}=A'\cup B'$ is a partition, if $I \times V(H)$ is evenly-split then it also satisfies 
\[ 
\alpha |H|\le |f^{-1}\left(I\times V(H)\right)\cap B'|\le \left(1-\alpha\right)|H|.
\]
If $I \times V(H)$ is not evenly-split then we say it is either \emph{$A'$-dominated} or \emph{$B'$-dominated} if
\[
|f^{-1}(I \times V(H)) \cap A'| \geq  \left(1-\alpha\right)|H| \qquad \text{ or } \qquad |f^{-1}(I \times V(H)) \cap B'| \geq  \left(1-\alpha\right)|H|,
\]
respectively. We then let $\chi(I)= 1$ if $I \times V(H)$ is $A'$-dominated, $\chi(I)= 2$ if $I \times V(H)$ is $B'$-dominated, and $\chi(I)= 3$ if $I \times V(H)$ is evenly-split (see Figure \ref{fig:sub1}). 

Let us suppose first that $|\chi^{-1}(3)| \geq \frac{|J|}{\rho t^{2}}$. In this case, at least $\frac{|J|}{\rho t^{2}}$ of the copies of $H$ are evenly-split. As mentioned in Section \ref{sec prelim_lemmas}, since $H(t-{z},s)$ is the product of $t-{z}$ copies of the complete graph $K_s$, using \eqref{e:Hammingiso}, we have that
\[
i(H(t-{z},s)) \geq \frac{1}{2} i(K_s) = \frac{s}{2}.
\]
In particular, each copy of $H$ which is evenly-split must contain at least $\frac{\alpha s |H|}{2}$ many edges from $A'$ to $B'$, and so it follows that there are at least
\[
\frac{|J|}{\rho t^{2}}\cdot\frac{\alpha s |H|}{2}  \geq \frac{\rho|M_{z}|}{t^{2}} \geq \frac{|W|}{t^{2}}
\]
paths of length 2 between $A'$ and $B'$.

So, let us assume that $|\chi^{-1}(3)| < \frac{|J|}{\rho t^{2}}$, and so the vast majority of copies of $H$ are either $A'$-dominated or $B'$-dominated. Furthermore, since 
\[
|A| \geq \frac{|W|}{3} \geq \frac{\rho|M_{z}|}{3}= \frac{\rho}{3} |J|\cdot|H|,
\]
it follows that $|\chi^{-1}(1)|,|\chi^{-1}(2)| \le \left(1-\beta\right)|J|$. Indeed, if $|\chi^{-1}(1)|>(1-\beta)|J|$, then we would have that
\begin{align*}
    |B| &\le |B'| < \beta|J|\cdot |H|+ \left(1-\beta\right)\alpha|J|\cdot |H|\\
    &=\frac{\rho}{3}|J|\cdot |H| = \frac{\rho |M_{z}|}{3}\le |A|,
\end{align*}
that is, $|B|<|A|$, contradicting our assumption that $|B|\ge |A|$. Similarly, if $|\chi^{-1}(2)|>(1-\beta)|J|$, we would have that $|A|<|A|$, a contradiction.

Hence, we may restrict our attention to the case where $|\chi^{-1}(1)|, |\chi^{-1}(2)|\le (1-\beta)|J|$ and $|\chi^{-1}(3)|<\frac{|J|}{\rho t^2}$. Since, by definition, $|\chi^{-1}(1)|+|\chi^{-1}(2)|+|\chi^{-1}(3)|=|J|$, we have that
\begin{align*}
    |\chi^{-1}(1)|,|\chi^{-1}(2)|\ge |J|-(1-\beta)|J|-\frac{|J|}{\rho t^2}\ge \frac{\beta}{2}|J|.
\end{align*}
Hence, by Theorem \ref{j iso}, there is a constant $c>0$ such that there are $c\sqrt{\frac{t}{z(t-z)}}\left(1-\frac{\beta}{2}\right)\frac{\beta|J|}{2}\ge \frac{|J|}{t}$ edges of $J$ which go between a vertex with colour $1$ and a vertex with colour $2$. Each of these edges represents an $A'$-dominated copy of $H$ (which we call the first copy of $H$) which is adjacent in $J$ to a $B'$-dominated copy of $H$ (which we call the second copy of $H$). Since at most $\alpha|H|$ of the vertices in the first copy are in $B'$ and at most $\alpha|H|$ of the vertices in the second copy are in $A'$, it follows that the number of vertices in the first copy which are in $A'$ and are matched to vertices in $B'$ in the second copy is at least
\[
(1-2\alpha)|H| = \left(1 - \frac{\rho}{2}\right) |H| \geq \frac{|H|}{2}, 
\]
since $\rho \leq 1$.

In particular, there are at least
\begin{align*}
    \frac{|J|}{t}\cdot \frac{|H|}{2}=\frac{|M_{z}|}{2t}\ge\frac{|W|}{2t}
\end{align*}
many edges in $G^2[M_{z}]$ between $A'$ and $B'$ and hence $\frac{|W|}{2t} \geq \frac{|W|}{t^2}$ many paths of length 2 between $A'$ and $B'$ in $G$.
\end{proof}

\begin{figure}[H]
\centering
\includegraphics[width=0.6\textwidth]{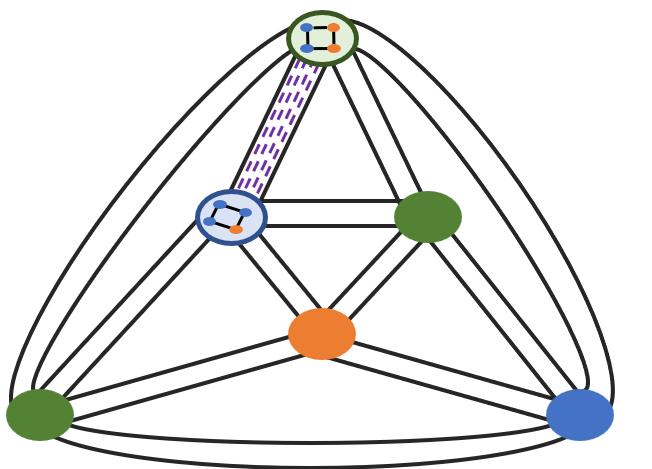}
\caption{A visualisation of the graph $G^2[M_{z}]$ in the case $t=4,z=2$ and $s=2$. Here, each circle is a copy of $H(t-z,s)$ and the black tubes represent the edges of $J(t,z)$. A pair of copies of $H(t-z,s)$ have been displayed, where the vertices of $M_{z}$ inside each copy have been coloured according to whether they lie in $A'$ (blue) or $B'$ (orange). Inside the black tube corresponding to the edge of $J(t,z)$ between this pair of vertices there is some matching, represented by the dashed (purple) edges. Note that this matching may not respect the structure of the $H$ copies in each circle. Each vertex $I$ of $J(t,z)$ is then coloured according to whether it is $A'$-dominated (blue), $B'$-dominated (orange) or evenly balanced (green), corresponding to $\chi(I)=1,2,$ and $3$, respectively.}
\label{fig:sub1}
\end{figure}

We note that, in terms of showing that the Erd\H{o}s-R\'{e}nyi component phenomenon does not hold in irregular product graphs, a simple application of Lemma \ref{t^k} is sufficient. Indeed, if the $G^{(i)}$ are all irregular and of bounded order, then it is not hard to show that there is some $\epsilon(C) >0$ such that $G$ must contain a vertex of degree at least $(1+2\epsilon)d$, and hence by Lemma \ref{t^k} at a subcritical probability $p=\frac{1-\epsilon}{d}$, with positive probability $G_p$ will contain a component of order $d^k$ for any fixed integer $k>0$, and so not all components in $G_p$ have linear order (in $d$) throughout the subcritical regime. However, Theorem \ref{many stars} shows that in the particular case of a product of many stars much more is true --- in parts of the subcritical regime \textbf{whp}  $G_p$ will even contain a component of order $|G|^{1-o_s(1)}$.

\section{Discussion}\label{s:discussion}
In what follows we will use $G$ to refer to the product graph, $t$ for its dimension, $d$ for its average degree, $\delta$ for its minimum degree and $C$ for the maximum degree of the base graphs. Let us briefly summarise the results on the typical component structure of percolated product graphs with probability close to $p=\frac{1}{d}$ which follow from Theorems \ref{product graphs} and \ref{regular product graphs}, and Theorems \ref{unbounded degree}, \ref{isoperimetric} and \ref{many stars}. There are three main assumptions which we work with:
\begin{description}[leftmargin=!,labelwidth=\widthof{\bfseries \textbf{(BO)}}]
\item[\textbf{(BO)}] The base graphs have bounded order;
\item[\textbf{(BD)}] The base graphs have bounded maximum degree;
\item[\textbf{(R)}] The base graphs are regular.
\end{description}
Note that clearly \textbf{(BO)} implies \textbf{(BD)}. 

Lichev \cite{L22} showed that \textbf{(BD)}, together with some mild isoperimetric assumptions, is sufficient to show the existence of a linear sized component in the supercritical regime, and to bound the order of the largest component in the subcritical regime. Whilst this bound on the order of the largest component in the subcritical regime was only polynomial in the size of the host graph, we note that Theorem \ref{many stars} implies that this bound is in fact close to optimal.

Indeed, whilst Theorem \ref{product graphs} implies that when $p=\frac{1- \epsilon}{d}$ \textbf{whp} the largest component of $G_p$ has order at most $\exp\left( - \frac{\epsilon^2 t}{9 C^2} \right) |G|$, by  Theorem \ref{many stars}, there is a choice of $G$ such that \textbf{whp} $G_p$ contains a component of order at least $\exp\left(-\Theta( t) \right) |G|$. In particular, \textbf{(BD)} alone is not sufficient to show that the percolated subgraph exhibits the Erd\H{o}s-R\'{e}nyi component phenomenon in the subcritical regime (which requires that the largest component in the subcritical regime is of logarithmic order in $|G|$).

Furthermore, Theorem \ref{isoperimetric} further strengthens Lichev's results in \cite{L22} by greatly weakening the isoperimetric assumptions and showing that at least in the supercritical regime the order of the largest component does behave quantitatively similar to the Erd\H{o}s-R\'{e}nyi random graph. As for the order of the second-largest component, Theorem \ref{isoperimetric} implies it is $o(|G|)$.
\begin{question}\label{q: BD poly 2nd comp}
Is there an example of a product graph $G$ which satisfies \textbf{(BD)}, and some mild isoperimetric assumptions, for which the second-largest component in the supercritical regime has polynomial order in $|G|$?
\end{question}
As demonstrated by Theorem \ref{regular product graphs}, \textbf{(BO)} and \textbf{(R)} are sufficient to guarantee that the percolated subgraph exhibits the Erd\H{o}s-R\'{e}nyi component phenomenon, and Theorem \ref{many stars} shows that \textbf{(R)} is necessary, in the sense that \textbf{(BO)} without the regularity assumption does not suffice to imply that the Erd\H{o}s-R\'{e}nyi component phenomenon holds. However, it remains an interesting open question as to whether \textbf{(BD)} and \textbf{(R)}, together with some mild isoperimetric assumptions, are themselves sufficient. For example, the following question is still open.
\begin{question}\label{q: BD and R}
Let $C>1$ be a constant, let $\alpha>0$ be a constant and let $G=\square_{j=1}^{t}G^{(j)}$ be a product graph where $G^{(i)}$ is regular, $1 \leq \left|d\left(G^{(i)}\right)\right| \leq C$ for each $i \in [t]$ and $i(G) \geq \alpha$. Does the phase transition that $G_p$ undergoes around $p= \frac{1}{d}$ exhibit the Erd\H{o}s-R\'enyi component phenomenon?
\end{question}

Note that Theorem \ref{unbounded degree} shows that at the very least \textbf{(BD)} is necessary to show the existence of a linear sized component in the supercritical regime, even under some quite strong isoperimetric constraints. It is perhaps interesting to ask what the limit of this pathological behaviour is. Indeed, the problem here seems to be an abundance of vertices of low degree (where the percolation is locally subcritical), which would suggest that at the very least a probability significantly larger than $p = \frac{1}{\delta}$ would be sufficient to show the existence of a linear sized component.

\begin{conjecture}\label{c:min deg perc}
Let $G$ be a high-dimensional product graph whose isoperimetric constant is at least inverse polynomial in the dimension, let $\epsilon >0$, let $\delta\coloneqq\delta(G)$ be the minimum degree of $G$ and let $p = \frac{1+\epsilon}{\delta}$. Then whp $G_p$ contains a component of size at least $(y(\epsilon) - o(1))|G|$, where $y(\epsilon)$ is defined according to \eqref{survival prob}.
\end{conjecture}

Finally, in the case of $G(d+1,p)$, Theorem \ref{ER thm} can be sharpened to show that, between the subcritical and the supercritical phases, there is a phase when the largest two components are both of polynomial order $|G|^{\gamma}$ for some $\gamma \leq \frac{2}{3}$, and then a phase when the largest component has order $|G|^{\gamma}$ for some $\gamma > \frac{2}{3}$ and the second-largest component has order at most $|G|^{\frac{2}{3}}$, although this whole transition happens in a very small window around $\frac{1}{d}$ (see, for example, \cite{FK16}).

It would be interesting to know how the component structure of the graph in Theorem \ref{many stars} develops as $p$ increases, and in particular at what point typically a unique giant component emerges, and when a linear sized component emerges.
\begin{question}\label{q:many stars transition}
Let $s$ be a sufficiently large integer and let $\epsilon>0$ be a sufficiently small constant, let $G^{(i)}=S(1,s)$ for every $1\le i \le t$ and let $G=\square_{i=1}^tG^{(i)}$. Let us write $L_1$ and $L_2$ for the largest and second-largest component, respectively, of $G_p$.
\begin{itemize}
\item For which $p$ is it true that \textbf{whp} $|L_2| = o(|L_1|)$? 
\item For which $p$ is it true that \textbf{whp} $|L_1| = \Theta(|G|)$?
\item Does there exist a probability $p$ and a constant $\gamma  >\frac{2}{3}$ such that \textbf{whp} both $|L_1|$ and $|L_2|$ have order at least $|G|^{\gamma}$?
\end{itemize}
\end{question}

Furthermore, as mentioned after the proof of Theorem \ref{many stars}, for \textit{any} family of bounded-order \textit{irregular} base graphs, the percolated product graph will not exhibit the Erd\H{o}s-R\'enyi phenomenon in the subcritical regime, containing a component of order at least a superlinear polynomial (in $d$) for certain subcritical values of $p$. However, in the particular case of Theorem \ref{many stars}, we in fact show that the largest component throughout some part of the sparse subcritical regime, where $p=\Theta(\frac{1}{d})$, is of polynomial order in $|G|$. It would be interesting to know if this behaviour is universal in high-dimensional products of bounded-order irregular graphs.
\begin{question} \label{q: generalised many stars}
For all $i\in [t]$, let $G^{(i)}$ be an irregular connected graph of order at most $C>0$. Let $G=\square_{i=1}^tG^{(i)}$. Let $\epsilon>0$ be a small enough constant, let $d\coloneqq d(G)$ be the average degree of $G$ and let $p= \frac{1-\epsilon}{d}$. Does there exist a $c(\epsilon,C)$ such that whp the largest component in $G_p$ has order at least $|G|^c$?
\end{question}

The table in Figure \ref{f:summary} summarises what is known about the order of the largest and second-largest component in $G_p$ for various values of $p$ and under various combinations of the assumption \textbf{(BD)}, \textbf{(BO)} and \textbf{(R)}, and indicates some combinations of probabilities and assumptions where open questions remain.

\begin{figure}[H]
\begin{table}[H]
\centering
\resizebox{\columnwidth}{!}{

\def\arraystretch{1.5}
\begin{tabular}{||c||c|c|c||} 
 \hline
 & $p=\frac{1-\epsilon}{d}$  & $p=\frac{1+\epsilon}{d}$ & $p=\frac{1+\epsilon}{\delta}$ \\ [0.5ex] 
 \hline\hline
  $\emptyset$ & -& 
  $\exists G$ with $|L_1|=o(|G|)$ (Theorem \ref{unbounded degree}) & Conjecture \ref{c:min deg perc}  \\
  \hline
  \textbf{(BD)} & $|L_1|\le \exp\left(-\frac{\epsilon^2t}{9C^2}\right)|G|$ (Theorem \ref{product graphs})& $|L_1|=\left(y(\epsilon)+o(1)\right)|G|$ (Theorem \ref{isoperimetric}) & - \\
  & $\exists G$ with $|L_1|\ge \exp\left(-\Theta( t) \right) |G|$ (Theorem \ref{many stars}) & $|L_2|=o(|G|)$ (Theorem \ref{isoperimetric}), Question \ref{q: BD poly 2nd comp} & \\
  \hline
  \textbf{(BD)} \& \textbf{(R)} & Question \ref{q: BD and R} & $|L_1|=\left(y(\epsilon)+o(1)\right)|G|$ (Theorem \ref{isoperimetric}) & - \\
  & & $|L_2|=o(|G|)$ (Theorem \ref{isoperimetric}), Question \ref{q: BD and R} & \\
  \hline
  \textbf{(BO)} \& \textbf{(R)} & $|L_1|=O(d)$ (Theorem \ref{regular product graphs}) & $|L_1|=\left(y(\epsilon)+o(1)\right)|G|$ (Theorem \ref{regular product graphs}) & - \\
  & & $|L_2|=O(d)$ (Theorem \ref{regular product graphs})& \\
  \hline
  \textbf{(BO)} \& $\neg$\textbf{(R)} & Question \ref{q: generalised many stars} & - & -  \\
  \hline
  \hline
\end{tabular}
}
\end{table}
\caption{A summary of the typical component structure of percolated high-dimensional product graphs}\label{f:summary}
\end{figure}

\bibliographystyle{plain}
\bibliography{perc}

\begin{thebibliography}{10}

\bibitem{AKS81}
M.~Ajtai, J.~Koml\'{o}s, and E.~Szemer\'{e}di.
\newblock Largest random component of a {$k$}-cube.
\newblock {\em Combinatorica}, 2(1):1--7, 1982.

\bibitem{AS16}
N.~{Alon} and J.~H. {Spencer}.
\newblock {\em {The probabilistic method}}.
\newblock Hoboken, NJ: John Wiley \& Sons, fourth edition, 2016.

\bibitem{B01}
B.~Bollob\'{a}s.
\newblock {\em Random graphs}.
\newblock Cambridge Studies in Advanced Mathematics. Cambridge University
  Press, Cambridge, second edition, 2001.

\bibitem{BKL92}
B.~Bollob\'{a}s, Y.~Kohayakawa, and T.~{\L}uczak.
\newblock The evolution of random subgraphs of the cube.
\newblock {\em Random Structures Algorithms}, 3(1):55--90, 1992.

\bibitem{BR06}
B.~{Bollob\'as} and O.~{Riordan}.
\newblock {\em {Percolation.}}
\newblock Cambridge University Press, Cambridge, 2006.

\bibitem{BH57}
S.~B. {Broadbent} and J.~M. {Hammersley}.
\newblock {Percolation processes. I: Crystals and mazes}.
\newblock {\em {Proc. Camb. Philos. Soc.}}, 53:629--641, 1957.

\bibitem{CEK13}
Demetres Christofides, David Ellis, and Peter Keevash.
\newblock An approximate isoperimetric inequality for {$r$}-sets.
\newblock {\em Electron. J. Combin.}, 20(4):Paper 15, 12, 2013.

\bibitem{CT98}
F.~R.~K. Chung and P.~Tetali.
\newblock Isoperimetric inequalities for {C}artesian products of graphs.
\newblock {\em Combin. Probab. Comput.}, 7(2):141--148, 1998.

\bibitem{DEKK22}
S.~Diskin, J.~Erde, M.~Kang, and M.~Krivelevich.
\newblock Percolation on high-dimensional product graphs.
\newblock {\em arXiv:2209.03722}, 2022.

\bibitem{D19}
R.~Durrett.
\newblock {\em Probability: theory and examples}.
\newblock Cambridge University Press, Cambridge, 2019.

\bibitem{ER59}
P.~{Erd\H{o}s} and A.{R\'enyi}.
\newblock {On random graphs. I}.
\newblock {\em {Publ. Math.}}, 6:290--297, 1959.

\bibitem{ER60}
P.~Erd\H{o}s and A.~R\'{e}nyi.
\newblock On the evolution of random graphs.
\newblock {\em Magyar Tud. Akad. Mat. Kutat\'{o} Int. K\"{o}zl.}, 5:17--61,
  1960.

\bibitem{ES79}
P.~Erd\H{o}s and J.~Spencer.
\newblock Evolution of the {$n$}-cube.
\newblock {\em Comput. Math. Appl.}, 5(1):33--39, 1979.

\bibitem{FK16}
A.~Frieze and M.~Karo\'{n}ski.
\newblock {\em Introduction to random graphs}.
\newblock Cambridge University Press, Cambridge, 2016.

\bibitem{G99}
G.~{Grimmett}.
\newblock {\em {Percolation.}}
\newblock Springer, Berlin, 1999.

\bibitem{HH17}
M.~Heydenreich and R.~van~der Hofstad.
\newblock {\em Progress in high-dimensional percolation and random graphs}.
\newblock CRM Short Courses. Springer, Cham; Centre de Recherches
  Math\'{e}matiques, Montreal, QC, 2017.

\bibitem{JLR00}
S.~Janson, T.~{\L}uczak, and A.~Ruci\'{n}ski.
\newblock {\em Random graphs}.
\newblock Wiley-Interscience Series in Discrete Mathematics and Optimization.
  Wiley-Interscience, 2000.

\bibitem{K82}
H.~{Kesten}.
\newblock {\em {Percolation theory for mathematicians}}.
\newblock Birkh\"auser, Boston, MA, 1982.

\bibitem{L22}
L.~Lichev.
\newblock The giant component after percolation of product graphs.
\newblock {\em J. Graph Theory}, 99(4):651--670, 2022.

\end{thebibliography}
\end{document}